\documentclass{amsart}
\usepackage{enumerate,amssymb,amsmath,graphicx,ucs, tikz, mathtools}
\usepackage[foot]{amsaddr}
\usepackage{hyperref}
\usepackage[cp1251]{inputenc}
\usepackage{times,mathabx,amsfonts,amssymb,amsrefs,mathrsfs,tikz, xcolor}
\usepackage{todonotes}
\usetikzlibrary{decorations,decorations.pathmorphing}

\usetikzlibrary{matrix}
\usetikzlibrary{math, patterns}

\numberwithin{figure}{section}

\newcommand{\thistheoremname}{}

\newtheorem*{genericthm*}{\thistheoremname}
\newenvironment{namedthm*}[1]
{\renewcommand{\thistheoremname}{#1}%
	\begin{genericthm*}}
	{\end{genericthm*}}
\newtheorem{thm}{Theorem}[section]
\newtheorem{prop}[thm]{Proposition}
\newtheorem{lemma}[thm]{Lemma}
\newtheorem{remark}[thm]{Remark}
\newtheorem{claim}[thm]{Claim}
\newtheorem{corol}[thm]{Corollary}

\theoremstyle{definition}
\newtheorem{definition}[thm]{Definition}
\newtheorem{rem}[thm]{Remark}
\newtheorem*{thm*}{Theorem}
\newtheorem*{definition*}{Definition}
\newtheorem*{lemma*}{Lemma}
\newtheorem*{statement*}{Claim}
\newtheorem{example}[thm]{Example}
\newtheorem*{corol*}{Corollary}
\newtheorem*{question*}{Question}

\def\q#1.{{\bf #1.}}

\def\Z{\mathbb Z}

\def\R{\mathbb R}
\def\H{\mathbb H}
\newcommand{\MM}{\operatorname{M}}
\newcommand{\OR}{\operatorname{OR}}
\newcommand{\ORI}{\operatorname{ORI}}
\newcommand{\opt}{\operatorname{opt}}
\newcommand{\diam}{\operatorname{diam}}
\newcommand{\mins}{\operatorname{Br}}
\newcommand{\Star}{\operatorname{Star}}
\newcommand{\Vol}{\operatorname{Vol}}
\newcommand{\Angle}{\operatorname{Angle}}
\newcommand{\Cone}{\operatorname{Cone}}

\long\def\comment#1{}

\author{Anna Erschler}
\address{A.E.: C.N.R.S., \'{E}cole Normale Superieur, PSL Research University, France}
\email{anna.erschler@ens.fr}

\author{Ivan Mitrofanov}
\address{I.M.: C.N.R.S., \'{E}cole Normale Superieur, PSL Research University, France }
\email{phortim@yandex.ru }

\title{Spaces that can be ordered effectively: virtually free groups and hyperbolicity}

\date{\today}

\subjclass[2010]{20F65, 20F67,20F69, 20F18}

\keywords{hyperbolic space, hyperbolic groups, 
	quasi-isometric invariants, uniform embeddings, 	space of ends, accessible groups}
\begin{document}
	\maketitle
	
	\begin{abstract}
		
		We study asymptotic invariants of metric spaces, defined in terms of the travelling salesman  problem, and our goal is  to classify groups and spaces depending on how well they can be ordered in this context.
		We characterize  virtually free groups as those  admitting an order which has some efficiency on
		$4$-point subsets.  
		We show that all $\delta$-hyperbolic spaces can be ordered extremely efficiently, for the question when the number of points of a subset tends to $\infty$. 
	\end{abstract}
	
	\section{Introduction}
	Given a metric space $(M,d)$, we consider  a finite subset $X$ of $M$.  We denote by  $l_{\opt}(X)$  the minimal length of a path which visits all points of $X$.
	Now assume that  $T$ 
	is an order on $M$ (here and in the sequel we always assume that the orders are total orders).
	For a finite subset $X \subset M$ we consider the restriction of the order
	$T$ on $X$, and enumerate 
	the points of $X$ accordingly:
	$$
	x_1 \leq_T x_2 \leq_T x_3 \leq_T \dots \leq_T x_{k}
	$$
	where $k = \#X$. Here and in the sequel  $\#X$ denotes the cardinality of the
	set $X$.
	We denote by $l_T(X)$ the length of the corresponding path
	$$
	l_T(X) := d(x_1, x_2) + d(x_2, x_3) + \dots + d(x_{k-1}, x_k).
	$$

	\begin{definition}
		Given an ordered metric space $\MM=(M, d,T)$ containing at least two points
		and $k\ge 1$, we define the {\it order ratio function}
		$$
		\OR_{M,T}(k) := \sup_{X\subset M | 2 \le \# X \le k+1} \frac{l_T(X)}{l_{\opt}(X)}.
		$$
	\end{definition}
	
	If $M$ consists of a single point, then the supremum in the definition above is taken over an empty set. 
	We use in this case the convention that $\OR_{M,T}(k)=1$ for all $k\ge 1$.

	Given an (unordered) metric space $(M,d)$, we also  define the {\it order ratio function} as 
	$$
	\OR_{M}(k) := \inf_{T} \OR_{M,T}(k).
	$$
	
	Given an algorithm to find an approximate solution of an optimization problem, the worst case of the ratio between the value provided by this algorithm
	and the optimal value is called the {\it competitive
		ratio} in computer science literature.
	We will recall below the setting of the universal travelling salesman  problem. 
	In this setting
	the function $\OR(k)$ corresponds to the competitive ratio for $k$-point subsets, in the situation when we do not require for a path to return to the starting point.
	
	It can be shown that even in the case of a finite metric space, where the infimum in the definition above is clearly attained, 
	the order that attains this minimum
	might depend on $k$ (see Remark \ref{rem:numberofpoints}).
	
	We say that a metric space $M$ is {\it uniformly discrete} if there exists $\delta>0$ such that for all pairs of points $x \ne y$ the
	distance between $x$ and $y$ is at least $\delta$. 
	It is not difficult  to see that the asymptotic class of the order ratio function behaves well with respect to quasi-isometries of
	uniformly discrete spaces, see Lemma \ref{le:qi}.
	
	This is no longer true without the assumption of uniform discreteness:  $M_1 = \mathbb{R}$ and 
	$M_2=\mathbb{R}\times [0;1]$ are in the same quasi-isometry class, but $M_1$ has a bounded order ratio function and the order ratio function of $M_2$ is unbounded; 
	indeed, the latter space contains a square. It is shown in \cite{HKL}
	that given  any order on the $n\times n$ square,  there exist a finite subset of this lattice, such that the length of the path, associated to this order, is at least  ${\rm Const} \sqrt[6]{\frac{\log n}{\log \log n}} $ multiplied  by the length of the optimal path. This implies
	an analogous bound in terms of the number of points
	of the subset $X$ rather than the size of the square:
	$\OR_{[0,1] \times [0,1]}(k) \ge  {\rm Const}\sqrt[6]{\frac{\log k}{\log \log k}}$.
	
	To avoid non-stability with respect to quasi-isometries, given a metric space $M$,  we can consider  order ratio
	functions for $(\varepsilon, \delta)$-nets of $M$ (see Definition \ref{def:net}).
	It is not difficult to see (see Lemma \ref{le:pullback}) that the function $\OR_{M'}$ is well-defined, up to a multiplicative constant,
	for $M'$ being an $(\varepsilon, \delta)$-net of $M$. 
	This allows us to speak about the {\it order
		ratio invariant} of a metric space, see Definition \ref{def:ORI} in Section \ref{sec:firstexamples}.

	In contrast with previous works on the competitive ratio of the universal travelling salesman  problem (\cites{bartholdiplatzman82, bertsimasgrigni, jiadoubling, HKL, Christodoulou, Schalekamp, bhalgatetal, gorodezkyetal, eades2}),
	we are interested not only in this asymptotic behaviour but also in particular values
	of $\OR(k)$. In the definition below we introduce the {\it order breakpoint} of an order :
	
	\begin{definition} \label{def:mins}
		Let $M$ be a metric space, containing at least two points,
		and let $T$ be an order on $M$.  
		We say that the order  
		breakpoint 
		$\mins(M,T) = s$ if $s$ is the smallest integer such that $\OR_{M,T}(s) < s$. 
		If such $s$ does not exist, we say that $\mins(M,T)=\infty$.
		In particular, $\mins(M,T)=2$ for a one-point space $M$.
	\end{definition}
	
	Given an (unordered)
	metric space $M$, we define $\mins(M)$
	as the minimum over all orders $T$ on $M$:
	$$
	\mins(M) = \min_{T}
	\mins(M,T).
	$$
	
	Given an order on $M$, the order breakpoint describes the minimal value $k$ for which  using this order as a universal order for the travelling salesman problem 
	has some efficiency on $(k+1)$-point subsets.

	From the definition, it is clear that $\mins(M,T) \ge 2$ for any $M$  and it is easy to see that $\mins(M,T)= 2$
	if $M$ is finite.  In Section \ref{sec:examples}, Theorem \ref{cor:finitegraphs} we evaluate the order breakpoint of finite graphs (containing edges): it is 2 for graphs homeomorphic to intervals, 3 for cactus graphs (see Definition \ref{def:cactus}) and 4 otherwise.
	In what concerns infinite graphs and metric spaces, it can be shown (see Lemma \ref{le:notrayline}) that a metric space quasi-isometric to a geodesic one has $\mins \ge 3$ unless it is bounded or quasi-isometric to a ray or to a line;  metric on the vertex set of a graph has order breakpoint equal to 2 if and only if the graph is bounded, or quasi-isometric to a ray or line.
	
	The order breakpoint is a quasi-isometric invariant for uniformly discrete metric spaces (see Lemma \ref{le:snakesqi}), hence it is well-defined for finitely generated groups. 
	We characterize finitely generated groups with small ($\le 3$) order breakpoint
	(in other words, groups that admit an order that has some efficiency on $4$-point subsets):
	
	\begin{namedthm*}{Theorem A} (= Thm \ref{thm:girth4})
		Let $G$ be a finitely generated group. 
		Then $G$ is virtually free if and only if $\mins(G) \le 3$ (in other words, if $G$ admits an order $T$ with  $\mins(G,T) \le 3$).
	\end{namedthm*}

	Using Stallings theorem 
	and Dunwoody's result about accessibility of finitely presented groups, we reduce the proof of  Theorem A to  Lemma \ref{le:oneended}, which provides a lower bound for $\mins$
	for one-ended groups, and to Lemma \ref{le:infprcycle} about
	mappings of 
	trimino graphs to infinitely presented groups. 
	
	It seems interesting to better understand  groups and spaces with given order breakpoint, in particular with $\mins =4$.
	We mention that there are examples of metric spaces with arbitrarily large order breakpoint (see Remark \ref{rem:graphsqiHd}).
	
	\begin{figure}[!htb] 
		\centering
		\includegraphics[scale=.45]{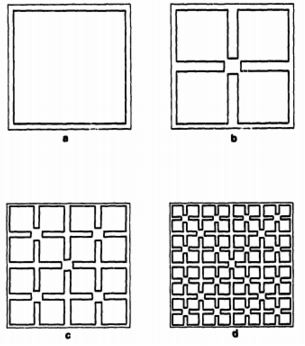}
		\caption{Bartholdi and Platzman used the order of the first appearance along the space-filling Sierpi\'{n}ski curve.}
		\label{pic:bartholdiplatzman}
	\end{figure}
	
	\bigskip
	
	Now we discuss the asymptotics of the order ratio function.	
	The travelling salesman problem is a problem to find a cycle of minimal total length that visits each of $k$ given points.
	Bartholdi and Platzman introduced the idea to order all points of a metric space and then, given a $k$-point subset,  visit its points in the corresponding order
	\cite{bartholdiplatzman82}, \cite{bartholdiplatzman89}.
	Such an approach is called the {\it universal travelling salesman problem}.
	(One of the motivations of  Bartholdi and Platzman was that this approach works
	fast for subsets  of a two-dimensional plane.) 
	Their  argument implies a logarithmic upper bound for the function $\OR_{\mathbb{R}^2}(k)$. 
	
	For some metric spaces the function $\OR$ is even better. 
	Namely, it is not difficult to show  that $\OR(k) \leq 2$, for all $k\ge 1$ in the case of metric trees, see Thm 2 \cite{Schalekamp}
	(an approriate order is shown on Figure \ref{pic:orderedtree}, see also Corollary \ref{re:anytree}). 
	
	\bigskip
	
	We prove that this best possible situation ($\OR(k)$ is bounded above by a constant) holds true for uniformly discrete
	hyperbolic
	spaces.
	
	\begin{namedthm*}{Theorem B} (= Thm \ref{thm:hyperbolicspace}). Let $M$  be a $\delta$-hyperbolic graph of
		bounded degree. 
		Then there exists  an order $T$ and
		a constant $C$
		such that for all $k\ge 1$
		$$
		\OR_{M,T}(k) \le C.
		$$
	\end{namedthm*}
	
	It is essential that the constant $C$ depends  on  $M$. 
	In contrast with trees, this constant does not need to be bounded by $2$ and it can be arbitrarily large.
	In particular, for subsets of $\H^d$ this constant
	depends on the dimension $d$ and the constants of uniform discreteness; 
	if we fix $\varepsilon$ and consider $(\varepsilon, \varepsilon)$-nets, then this constant tends to infinity as the dimension $d$ of $\mathbb{H}^d$ tends to infinity. 
	In contrast with trees where the statement is true for any metric tree, it is essential in general to assume that the degree of the graph
	$M$ is bounded.

	It is not difficult to see that  some non-hyperbolic metric spaces also have bounded order ratio function; a particular family that arises from gluing together trees or hyperbolic spaces is studied in
	Subsection \ref{Subsection:nothyperbolic}.
	So far we do not know whether a finitely generated non-hyperbolic group can have bounded $\OR(k)$; for more related questions see the discussion at the end
	of Subsection \ref{Subsection:nothyperbolic}.

	We continue our study of the order ratio function and the order breakpoint in our subsequent paper \cite{ErschlerMitrofanov2}, where we show  their relation with Assouad-Nagata dimension.

	\subsection{Plan of the paper}
	
	In Section $2$ we discuss basic properties of the order ratio  function and the order breakpoint, their relation to quasi-isometries,  uniform mappings and the notion of ``snakes''.
	
	In Section 3 we start 
	with elementary examples of metric trees. We describe the asymptotics of the order ratio
	function for compositions of wedged sums of metric spaces.
	As a particular case we characterise  the values of $\mins$ of finite graphs.
	
	In Section 4 we study metric spaces and groups with small order breakpoint. 
	The already mentioned Lemma \ref{le:notrayline} studies spaces with $\mins = 2$. 
	Its proof -- which uses a result of M.~Kapovich about what are called $(R,\varepsilon)$-tripods -- is given in Appendix $B$.
	The main goal of the section is  to prove  Theorem \ref{thm:girth4} and thus characterise groups with $\mins \le 3$.
	
	The goal of Section 5 is the proof of  Theorem B. Since  by a theorem of Bonk and Schramm \cite{bonkschramm} any $\delta$-hyperbolic space of bounded geometry can be  quasi-isometrically embedded into $\H^d$, taking in account  Lemma \ref{le:qi} about quasi-isometric embeddings,
	the main goal in the proof of Theorem~B is to prove it for an $(\varepsilon, \varepsilon)$-net of $\H^d$. 
	We do it by choosing an appropriate tiling
	of this space and considering a naturally associated tree; there is a family of {\it hierarchical orders} on this tree, and we choose any of these orders. 
	While it is well-known
	that finite subsets of hyperbolic spaces can be well approximated by metric trees, we would like to stress that such an approximation does not provide a priori any upper bound for the order ratio function.
	We control the total length of piecewise-geodesic continuous paths associated to this order  by bounding the number of tiles that are visited by such paths (here we use  bounds on cones of a continuous path in a hyperbolic space, Lemma \ref{le:cone}), and  bounding the number of visits
	(see Lemma  \ref{lem:multiplicity}). 
	
	\begin{figure}[!htb] 
		\centering
		\includegraphics[scale=.75]{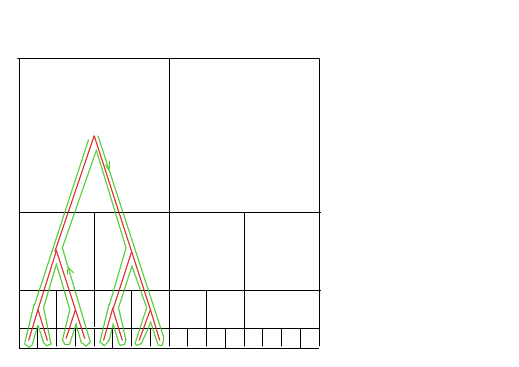}
		\caption{An order on the tiles of a hyperbolic plane $\H^2$ 
			corresponding to the first visit of a tile by a green path.
			For $\H^d$, $d\ge 2$,  an order can be chosen to be 
			hierarchical  on an appropriate tree.
		}
		\label{pic:tilingwithtree2}
	\end{figure}

	In subsection \ref{Subsection:nothyperbolic} we study (not necessarily hyperbolic) families of metric spaces with bounded $\OR$.

	{\bf Acknowledgements. } 
	We would like to thank 	Chien-Chung Huang and Kevin Schewior
	for references  on the travelling salesman  problem,  and Nick Ramsey for conversations on first order logic.
	This project has received funding from the European Research Council (ERC) under
	the European Union's Horizon 2020 
	research and innovation program (grant agreement
	No.725773).
	
	{\bf Data availability. } 
	Data sharing not applicable to this article as no datasets were generated or analysed during the current study.
	
	\section{Preliminaries and basic properties}\label{sec:firstexamples}
	
	In this section we describe basic properties of the order ratio function and order breakpoint.
	
	\begin{lemma}
		\label{prop:max_min_or}
		For any ordered metric space $(M, d, T)$ and any  $k \ge 1$
		it holds that
		$$
		1 \leq \OR_{M,T}(k) \leq k.
		$$
	\end{lemma}
	\begin{proof}
		The first inequality is obvious. 
		Let $X \subset M$, $\# X = k+1$ and let $L$ be the diameter of $X$. 
		Then $l_{T}(X) \leq k L$ and $l_{\opt}(X) \geq L$.   
	\end{proof}

	\begin{definition} \label{de:qi} Given metric spaces $N$ and $M$, a map $\alpha$  from $N$ to $M$ is a {\it  quasi-isometric embedding} if
		there exist $C_1 \geqslant 1$, $C_2 \geqslant 0$ such that for any $x_1, x_2 \in N$ it holds that
		\[
		\frac{1}{C_1}(d_N(x_1, x_2)) -C_2  \le   d_M(\alpha(x_1), \alpha(x_2)) \le C_1(d_N(x_1, x_2)) +C_2
		\]
		
		If $M$ is at bounded distance from $\alpha(N)$, this map is called a quasi-isometry, and the spaces $X$ and $Y$ are called {\it quasi-isometric.}
		If $\alpha$ is bijective and $C_2=0$, then the spaces are said to be  {\it bi-Lipschitz equivalent}.
	\end{definition}
	
	A significantly weaker condition is the following.
	
	\begin{definition}\label{de:ue}  Given  metric spaces $N$ and $M$, a map $\alpha$  from $N$ to $M$ is an
		{\it uniform embedding} (also called {\it coarse embedding})
		if there exist two non-decreasing functions $\rho_1$, $\rho_2:$ $[0, +\infty) \to [0, + \infty)$,
		with $\lim_{r\to \infty} \rho_1(r) = \infty$,  such that
		\[
		\rho_1 (d_X(x_1, x_2)) \leqslant d_M(\alpha(x_1), \alpha(x_2)) \le \rho_2 (d_X(x_1, x_2)).
		\]
	\end{definition}
	
	\begin{definition}\label{def:epsnet}
		We say that for $\varepsilon> 0$ a subset $U$ in a metric set $M$ is an $\varepsilon${\it -net} if any point of $M$ is at distance at most
		$\varepsilon$ from $U$.
	\end{definition}
	
	\begin{definition}\label{def:net}
		We say that for $\varepsilon, \delta > 0$ a subset $U$ in a metric set $M$ is an $(\varepsilon, \delta)${\it -net} if any point of $M$ is at distance at most
		$\varepsilon$ from $U$ and the distance between  any two distinct points of $U$ is at least $\delta$.
	\end{definition}
	
	It is clear that any $\varepsilon$-net of $M$ (in particular, any $(\varepsilon, \delta)$-net of $M$),  endowed with the restriction of the metric of $M$, is quasi-isometric to 
	$M$. 
	Indeed, its embedding to $M$ satisfies the condition of Definition \ref{de:qi} with $C_1=1$ and $C_2=0$. 
	Observe also  that $M$  is at distance of most $\varepsilon$ from  the image of this embedding.
	It is also   clear from the definition that any $(\varepsilon, \delta)$-net is uniformly discrete.
	For convenience of the reader, we recall the following well-known observation.
	
	\begin{rem} \label{re:separatednet}
		
		Let $M$ be a metric space. For any  $\varepsilon >0$ there exists an 
		$( \varepsilon, \varepsilon)$-net of $M$.
		
	\end{rem}
	
	\begin{proof}
		Consider a maximal, with respect to inclusion, subset $U$ of $M$ such that the distance between any two points is at least $ \varepsilon$ (in other words $U$ is a maximal with respect to inclusions $\varepsilon$-separated net of $M$).   The maximality of this set guarantees that any point $x \in M$ is at distance at most $\varepsilon$ from $M$.
	\end{proof}

	\begin{definition}
		Given (a not necessary injective) map $\phi: N \to M$ and an order $T_M$ on $M$ let us say that an order $T_N$ 
		on $N$ is a {\it pullback} of $T_M$ if the following holds: for any $x , x' \in N$ such that  $\phi(x) <_{T_M} \phi(x')$  we have
		$x<_{T_N} x'$. 
	\end{definition}
	
	Note that a pullback of an order always exists: to construct a pullback it is (necessary and) sufficient to fix an order on each preimage $\phi^{-1}(y)$ for $y\in M$.
	
	\begin{lemma}\label{le:pullback}
		Let $\alpha$ be a quasi-isometric embedding of a uniformly discrete space $N$ to a uniformly discrete space $M$, then there is a constant $K$ such that if \: $T$ is an order on $M$ and $T_N$ is its pullback on $N$, then for all $k \ge 1$ 
		\[
		\OR_{N,T_N} (k) \le K \cdot \OR_{M,T} (k).
		\]
		The constant $K$
		can be chosen depending only on the quasi-isometry constants of $\alpha$ and the uniform discreteness constants of $M$ and $N$.
	\end{lemma}
	
	\begin{proof}
		There exists $\delta > 0$ such that all distances between distinct points in $N$ and $M$ are greater than $\delta$. 
		Since $\alpha$ is a quasi-isometry  and  since $N$ and $M$ are  uniformly  	discrete, we know  that there exists $C$ such that
		\[\frac{1}{C} d_{M}(\alpha(x_1), \alpha(x_2)) \le  d_{N}(x_1, x_2) \le C d_{M}(\alpha(x_1), \alpha(x_2))
		\]
		if $\alpha(x_1) \neq \alpha(x_2)$, and $d_N(x_1, x_2) < C$ if $\alpha(x_1) = \alpha(x_2)$.
		
		Let $A \subset N$, $\# A = k+1$, $B = \alpha(A)$, $\# B = n+1$, $n \le k$.
		Let $L_1= l_{\opt}(B)$, $L_2=l_{T}(B)$.
		We know that $L_2 \leqslant L_1 \cdot \OR_{M,T}(n) \leqslant L_1 \cdot \OR_{M,T}(k)$, because the $\OR$ function is non-decreasing.
		If some path goes through all points  of $A$, the corresponding path in $M$ goes through all points of $B$. 
		Hence $l_{\opt}(A) \ge \frac{1}{C}L_1$.
		The path that enumerates $A$ according to the order $T_N$ consists of $n-1$ preimages of edges of $l_T(B)$ and $k-n$ edges with one point image.
		The first group of edges gives a contribution of $\le C L_2$ towards $l_{T_N}(A)$, and the second group of edges gives $\leq C(k-n) \leq C k$.
		We also know that $k < \frac{l_{\opt}(A)}{\delta}$.
		Putting these inequalities together, we obtain
		\[l_{T_N}(A) \le C^2 \cdot \OR_{M,T}(k) l_{\opt}(A) + \frac{C}{\delta} l_{\opt}(A).
		\]
		Therefore
		\[
		\OR_{X,T_N}(k) \leq C^2 \cdot \OR_{M,T}(k) + \frac{C}{\delta},
		\]
		and we can take $K(\alpha) = C^2 + \frac{C}{\delta}$.
	\end{proof}
	
	Lemma \ref{le:pullback} above  allows  to speak about the order ratio invariant of a metric space.
	
	In the definition below we say that two functions $f_1,f_2 :\mathbb{N} \to \mathbb{R}$ are equivalent up to a multiplicative constant if  there exists $K_1,K_2>0$ such that $f_1(r) \le K_1 f_2(r)$ and $f_2(r) \le K_2 f_1(r)$  for all $r\ge 1$.

	\begin{definition}\label{def:ORI}
		For a metric space $(M, d)$ its {\it order ratio invariant} $\ORI_{M}$ 
		is defined to be the equivalence class, with respect to equivalence up to a multiplicative constant, 	of functions  $\OR_{M'}$, where $M'$ is some $(\varepsilon, \delta)$-net of $M$.
	\end{definition}
	
	In view of Lemma \ref{le:pullback} we obtain the following.
	
	\begin{lemma} \label{le:qi}
		For any metric space $M$, $\ORI_M$ is well-defined, and it is a quasi-isometric invariant of a metric space.	
	\end{lemma}
	
	\begin{proof}
		If $M_1$ and $M_2$ are $(\varepsilon, \delta)$-nets of $M$, they are quasi-isometric and from Lemma \ref{le:pullback} it follows that $\OR_{M_1}$ and $\OR_{M_2}$ are equivalent up to a multiplicative constant.
		If two spaces are quasi-isometric then their $(\varepsilon, \delta)$-nets are also quasi-isometric.
		
	\end{proof}
	
	While specific values of the order ratio function can change under quasi-isometries, as can be easily seen already  from finite space examples, 
	the equality  $\OR(s) = s$ for given $s$ (the maximal possible value, see Lemma \ref{prop:max_min_or}) is preserved under quasi-isometries, as we will see in Lemma \ref{le:snakesqi}.
	
	Consider a sequence of  $s$  points: $x_1 <_T x_2 <_T \dots <_T x_s$.
	
	\begin{definition} \label{def:snake}
		We call such sequence to be  {\it a snake} on  $s$ points, of diameter $a$ and of  {\it width} $b$,
		if the diameter of this set is $a$ and the maximal  distance $d(x_i, x_j)$ where $i$ and $j$ are of the same parity, is $b$.
		For a snake on at least two points, 
		we say that the ratio $a/b$ is the {\it elongation} of the snake.
		If $s=2$, then the width of the snake is $0$, and we say in this case that the elongation is infinite.
	\end{definition}

	\begin{lemma}\label{lem:snake} Let $M$ be a metric space, containing at least two points, 
		and let $T$ be an order on $M$ and consider
		$s\geq 1$. 
		The following conditions are equivalent:
		\begin{enumerate}
			\item $\OR_{M,T}(s) = s$
			\item There exist snakes on  $s+1$ points in $(M,T)$ of arbitrarily large elongation.
		\end{enumerate}
	\end{lemma}
	\begin{proof}
		$(2)\to (1)$. Let $S = \{x_1, \dots, x_{s+1}\}$ be a snake of diameter  $a$ and width $b$.  
		From the triangle inequality it follows that the distance between any consecutive points $d(x_i, x_{i+1})$ is greater than or equal to  $a - 2b$.  This shows that  $l_T(S) \geqslant s(a - 2b)$. 
		Observe also that $l_{\opt} (S) \leqslant a + (s-1)b$, because we can pass first through
		all points with odd indexes and then all points with even indexes.
		Hence
		\[\OR_{M,T}(s) \geqslant \frac{s(a-2b)}{a + (s-1)b}
		\]
		that tends to $s$ as $b/a$ tends to $0$.
		The other inequality $\OR_{M,T}(s) \le s$ comes from Lemma \ref{prop:max_min_or}.
		
		\medskip
		\noindent$(1)\to (2)$. 
		We can assume that $s \geq 2$.
		Let $0 < \varepsilon < 1/(2s+2)$, $S = \{x_1, \dots, x_{s+1}\}$  and
		$l_{T}(S) \ge s(1 - \varepsilon) l_{\opt} (S)$.
		
		Since $l_T(S) \leqslant s\diam(S)$ and 
		$l_{\opt} \geqslant \diam(S)$, we conclude that 
		\begin{itemize}
			\item $s\diam(S)(1-\varepsilon) \leqslant l_T(S)\leqslant s\diam(S)$ and
			\item $\diam(S) \leqslant l_{\opt}(S) \leqslant \diam(S)/(1-\varepsilon)$.
		\end{itemize} 
		The first inequality implies that for all $i$, 
		$d(x_i, x_{i+1})$ is close to the diameter of $S$: $d(x_i, x_{i+1}) \ge \diam(S)(1-s\varepsilon)$.
		
		The second inequality implies that for any triple of points of  $S$, the minimal sum of two pairwise distances  is at most  $\diam(S)/(1-\varepsilon)$.
		If $\varepsilon < 1/2$, then 
		$1/(1-\varepsilon) < 1 + 2\varepsilon$.
		
		Applying these inequalities for the triple
		$(x_{i},x_{i+1},x_{i+2})$, we observe that
		$d(x_i, x_{i+2}) < (2+s)\varepsilon \diam(S) \leq 2s\varepsilon \diam(S)$.
		Observe that the shortest path joining $x_i$, $x_{i+1}$, $x_{i+2}$ can not be the path visiting them in this order $x_i$, $x_{i+1}$, $x_{i+2}$. 
		Indeed, the sum $d(x_i, x_{i+1}) + d(x_{i+1}, x_{i+2})$ can not be the minimal sum of two pairwise distances from the triple, because $d(x_i, x_{i+1}) + d(x_{i+1}, x_{i+2}) \ge \diam(S)(2-2s\varepsilon) > \diam(S)(1 + 2\varepsilon)$.

		This implies  that the width of the snake $S$ is at most 
		$2s^2\varepsilon \diam(S)$,
		and its elongation is greater than of equal than  $1/(2s^2\varepsilon)$.
	\end{proof}
	
	Since a snake on $k$ points contains snakes on fewer points, it is clear that 
	$\OR_{M,T}(k) = k$ 
	implies  $\OR_{M,T}(a) = a$ for all integers $a < k$.
	
	\begin{lemma}\label{le:snakesqi}[Snakes, Order Breakpoint and quasi-isometries]
		\begin{enumerate}
			\item
			Let an integer $k$ be  $\ge 2$, let $N$ be a uniformly discrete space with $\OR_N(k) = k$ and let $\alpha$ be a quasi-isometric embedding of $N$ to a metric space $M$. 
			Then $\OR_M(k) = k$.
			\item In particular, if $N$ and $M$ are quasi-isometric uniformly discrete metric spaces, then 
			$\mins(N) = \mins(M)$.
		\end{enumerate}
	\end{lemma}
	
	\begin{proof} 
		(1) Observe that if $\OR_N(k)=k$, then by definition of order ratio function we have $\OR_{N,T}(k)=k$ for any order $T$ on $N$.
		If $\OR_M(k) \ne k$, then there exist an order $T_M$ such that $\OR_{M,T_M}(k) < k$. Let $T$ be a pullback of $T_M$. 		
		Since $\OR_{N,T}(k)=k$, we know that  $(N,T)$ admits a sequence of snakes on $k+1$ points with elongations that tend to infinity.
		
		A snake $S$ on $k+1$ points in $(N,T)$ with large enough elongation  maps to a snake $\alpha(S)$ on $k+1$ points in $(M, T_M)$, 
		and in particular these points are distinct. The argument is the same as in the beginning of the proof of Lemma \ref{lem:uniform} mutatis mutandis.
		
		Since $\alpha$ is a quasi-isometry, the length of $\alpha(S)$ is bounded from below by some affine function of the length of $S$, and the width of $\alpha(S)$ is bounded from above by some affine function of the width of $S$.
		
		Because of the uniform discreteness, the widths of these snakes in $(N,T)$ are uniformly separated from zero, so their diameters tend to infinity. 
		The diameter of $\alpha(S)$ is bounded  below by some affine function of the diameter of $S$, 
		and the width of $\alpha(S)$ is bounded above by some affine function of the width of $S$.
		As $N$ is uniformly discrete, and by restricting to snakes $S$ of large diameter, we can choose these affine functions to have ho additive constants. Since the elongations of the snakes $S$ tend to infinity, it follows that the elongations of the snakes $\alpha(S)$ tend to infinity, and this is in contradiction with the fact that   $\OR_{M,T_M}(k) < k$.
		
		(2) Follows from (1), because $\mins(M)$ is the smallest integer $s$ such that $\OR_M(s) < s$.
	\end{proof}

	Now we introduce a strengthening of the condition $\OR(k)=k$. 
	In contrast with the property of admitting
	snakes of large elongation, this property will be inherited by the images 
	of uniform embeddings.
	\begin{definition} \label{def:strongsnakes}
		We say that an ordered metric space $(M,T)$ admits {\it an $(\infty, b)$-sequence of snakes} on $s$ points if it admits a sequence of
		snakes on $s$ points, in which the diameters of snakes tend to $\infty$ and the widths of all snakes are $\le b$.
		If a sequence of snakes on $s$ points is an $(\infty, b)$-sequence for some $b$, we say that this sequence of snakes is an $(\infty, \rm{Bounded})$-sequence. 
	\end{definition}
	
	This notion allows us to obtain estimates for $\mins$ in the case when one metric space is uniformly embedded in another one.
	
	\begin{lemma} \label{lem:uniform}
		Let $N$ and $M$ be metric spaces.
		Let $\alpha$ be a uniform embedding of $N$ to $M$, $T_M$ be an order on $M$, and $T_N$ be its pullback on $N$. 
		Then if for some $k\ge 1$ the ordered space $(N,T_N)$ admits an $(\infty, \rm{Bounded})$-sequence of  snakes on $k+1$ points, then $(M,T_M)$ also admits such sequence of snakes and
		$\OR_{M,T_M}(k)= k$.
	\end{lemma}
	
	\begin{proof} 
		Let $K$ be such that any two points at distance $\geq K$ in $N$ map to a pair of distinct points in $M$.
		Let $X=(x_1 <_{T_N} \dots <_{T_N} x_{k+1})$ be a snake in $N$ of diameter $a$ and width $b$ and let $a > 2b + K$. 
		Note that $\alpha(x_i)\neq \alpha(x_j)$ for any $i$ and $j$ of different parity.
		This means that 
		$$
		\alpha(x_1) <_{T_M} 
		\dots <_{T_M} \alpha(x_{k+1}), $$ 
		hence $\alpha(X)$ is a snake on $k+1$ points in $M$.
		
		Consider  an $(\infty, \rm{Bounded})$-sequence of  snakes in $(N, T_N)$.
		Observe that its image under a uniform embedding is a sequence of  $(\infty, \rm{Bounded})$-snakes in $(M, T_M)$.
		We can conclude therefore that $M$ admits snakes on $k+1$ points of arbitrarily large elongation, and hence by Lemma \ref{lem:snake} we know that 
		$\OR_{M,T_M}(k)= k$.
	\end{proof}
	
	The following lemma implies that the order ratio function of a metric space  is defined by the  order ratio functions of all its finite subsets.
	
	\begin{lemma}\label{rem:goedel}
		Let $M$ be a metric space. Consider a  function $F:\mathbb{Z}_+ \to \mathbb{R}_+$ and assume that for any finite subset $M' \subset M$ there exists an order $T'$ satisfying
		$\OR_{M', T'} (k) \le F(k)$ for all $k\ge 1$.
		Then there exists an order $T$ on $M$ satisfying 
		$\OR_{M, T} (k) \le F(k)$ for all $k\ge 1$.
	\end{lemma} 
	
	\begin{proof}
		
		Indeed, we can take one constant symbol for each point of $M$.
		Observe that for a binary predicate ``$<$'', the fact that it defines an order  can be described by first order sentences.
		Also for any finite subset $X\subset M$ we forbid all its orders $T_X$ for which $\OR_{X,T_X}(k) > F(k)$ for some $k$, this can also be described by a set of first order formulas. 
		Then the statement of this lemma follows from the Compactness theorem.
	\end{proof}
	
	Note that for any finite metric space its $\mins$ is equal to $2$. 
	This means that the inequality $\mins(M) \leq s$ does not follow from inequalities $\mins(M') \leq s$ for finite subsets of $M$.
	To prove that $\mins(M)\leq s$ you  also have to provide for all finite $M'\subset M$ some uniform bound for elongations of snakes on $s+1$ points.
	
	\medskip
	
	Now we give one more definition that will be used in later sections.
	
	\begin{definition}\label{def:convex}
		Let $M$ be a set, $T$ be an order on  $M$ and $A\subseteq M$ be a subset.
		We say that $A$ is \emph{convex} with respect to  $T$ if for all  $x_1, x_2, x_3 \in M$ such that  
		$x_1 <_T x_2 <_T x_3$ it can not be that $x_1, x_3 \in A$ and $x_2 \not \in A$.
	\end{definition}

	We give a lemma that 
	provides a sufficient condition for the existence of an order where the sets of some given family are ``convex''.
	
	\begin{lemma} \label{le:convex} 
		Let  $\mathcal{A}$ be a family of subsets of  $M$ such that for
		any two sets $A_1, A_2 \in \mathcal{A}$ either their intersection is empty, 
		or one of the sets is contained in another one.
		Then there exists an order $T$ on $M$ such that all sets in $\mathcal{A}$ are convex with respect to $T$.
	\end{lemma}

	\begin{proof}
		For finite $M$ the claim can be proved by induction on the cardinality of $M$.
		If $\# M = 0$ (and hence $M=\emptyset$) then the statement is obvious.
		Inductive step: if there are no sets besides $\emptyset$ and $M$ in $\mathcal{A}$, then any order satisfies the desired conditions.
		Otherwise, let a set 
		$B_1 $  be maximal by inclusion in $\mathcal A\setminus \{M\}$ (we use the same letter $M$ for the space $M$ and here for the set containing all points of $M$)
		and let $B_2 = M\setminus B_1$.
		Our assumption on $\mathcal{A}$ and the maximality of $B_1$ implies that 
		any $A\in \mathcal A$, 
		$A \ne M$, $A \ne \emptyset$  is either a subset of $B_1$ or a subset of $B_2$.
		Since $\#B_1, \#B_2 < \#M$, the induction hypothesis implies the existence of orders $T_1$ on $B_1$ and $T_2$ on $B_2$ such that for $i\in \{1,2\}$ all sets of $\mathcal{A}$ that are subsets of $B_i$ are convex with respect to $T_i$.
		The required order $T$ can be constructed as follows: all elements of $B_1$ are $<_T$ than all elements of $B_2$, and $T = T_i$ on $B_i$ for $i=1$ or $2$.
		
		Now consider the case when $M$ is not necessarily finite. We use again the Compactness theorem (similarly to the proof of Lemma \ref{rem:goedel}) to make a reduction to the finite case. 
		Consider constant symbols corresponding to points of $M$ 
		and to the  binary relation $<_T$.
		Observe that the axioms of a relation  to be an order and the condition in the definition of convex sets 
		can be described by sentences of first order. 
		Any finite collection of these sentences has a model by applying the lemma to finite subsets of $M$, hence the whole collection of sentences has a model.
	\end{proof}
	
	Now we mention some examples of families of convex sets for metric spaces which we describe in the following sections.
	\begin{enumerate}
		\item For rooted trees we can choose an order such that all branches of the tree are convex with respect to the order (see Lemma \ref{lem:hierarchicaltrees}.)
		
		\item For the graph $\Gamma_d$ associated to a tiling of a hyperbolic space $\H^d$ and the order, constructed in Section \ref{sec:hyper}, the sets that correspond to branches of the corresponding tree are convex.
		
		\item In Definition \ref{def:star} we define Star orders for graphs endowed with edges. 
		Star figures (union of a vertex and all the outgoing half-edges) are convex with respect to these orders.
	\end{enumerate}
	
	In these examples we could find the orders by showing  that the above mentioned sets satisfy the  assumption of Lemma \ref{le:convex} and apply this lemma. 
	We do not explain the details, since, as we have mentioned, we will describe these orders more explicitly in  Sections \ref{sec:examples}, \ref{sec:hyper}.
	
	We already mentioned an order on a square, constructed by Bartholdi and Platzman.  
	Now we describe explicitly another hierarchical  order on the unit square.
	
	This order also uses self-similarity in its construction, 
	and we chose this particular order to guarantee that the squares (discussed below) are convex.
	
	\begin{example} Let $M =[0;1)^2$.
		For each point $x = (x_1,x_2)\in M$ we consider the binary representations $0.a_1a_2\dots$ and $0.b_1b_2\dots$ of  $x_1$ and $x_2$ correspondingly. Given $x$, we define $c(x)$ to be the number with binary representation $0.a_1b_1a_2b_2\dots$.
		We define the order $T$ by setting $x <_T y \Leftrightarrow c(x) < c(y)$.
		Then all parts that are obtained by splitting $M$ into $2^{2k}$ equal squares are convex with respect to $T$.
		
	\end{example}	
	The figure below illustrates the order on these smaller squares.
	
	\begin{figure}[h!]
		
		\begin{tikzpicture}
		\draw[thick](0,0) -- (2,0) -- (2,2) -- (0,2) -- cycle;
		\draw[thick](1,0) -- (1,2);
		\draw[thick](0,1) -- (2,1);
		\draw[thick][-stealth](0.5,0.5) -- (0.5,1.5) -- (1.5,0.5) -- (1.5,1.5);
		
		\begin{scope}[shift={(3,0)}, scale = 0.5]
		\draw[thick](0,0) -- (4,0) -- (4,4) -- (0,4) -- cycle;
		\draw[thick](1,0) -- (1,4);
		\draw[thick](2,0) -- (2,4);
		\draw[thick](3,0) -- (3,4);
		\draw[thick](0,1) -- (4,1);
		\draw[thick](0,2) -- (4,2);
		\draw[thick](0,3) -- (4,3);
		\draw [-stealth] (0.5,0.5) -- (0.5,1.5) -- (1.5,0.5) --
		(1.5,1.5) -- (0.5,2.5) -- (0.5,3.5) -- (1.5,2.5) --
		(1.5,3.5) -- (2.5,0.5)  -- (2.5,1.5) --  (3.5,0.5) -- (3.5,1.5) --
		(2.5,2.5) -- (2.5,3.5) -- (3.5,2.5) -- (3.5,3.5);
		\end{scope}
		\end{tikzpicture}
		\caption{Each of the small squares is convex with respect to $T$. The picture illustrates the induced order on the set of squares.}
	\end{figure}
	
	\section{ Metric trees, acyclic gluings, finite graphs}\label{sec:examples}
	
	In this section we 
	describe first examples of evaluation of the order breakpoint and of asymptotic behaviour  of the order ratio function.
	Some of these examples will be used in the following sections.
	
	\begin{figure}[!htb] 
		\centering
		\includegraphics[scale=.35]{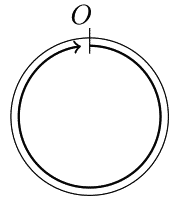}
		\caption{The clockwise order $T_0$ on the circle.}
		\label{pic:circleorder}
	\end{figure}
	
	\begin{lemma} \label{le:examplecircle} Let $M=S^1$ be a circle (with its inner metric). 
		Then $\OR_M(k)=2$ for all $k \geqslant 2$. More precisely, for any order $T$ of $S^1$ and any $\varepsilon > 0$ there exists a snake on $3$ points such that its 
		points are located in $\varepsilon$-neighborhoods of two antipodal points of $S^1$.
	\end{lemma}
	
	\begin{proof} 
		Denote by $2R$ the length of the circle.
		Let $T_0$ be a natural  order  on this circle (see Figure \ref{pic:circleorder}).
		Suppose that $X \subset M$ and $l_{\opt}(X) = a$. 
		This means that the set $X$ lies on an arc $AB$ of length $a$. 
		If this arc does not contain the point $O$, then 
		$l_{T_0}(X) = a$. If this arc contains $O$, then $l_{T_0}(X) \leqslant 2a$.
		Hence $\OR_M(k) \le 2$ for all $k$.
		
		Since the function $\OR_M(k)$ is non-decreasing in $k$, it is enough to prove that 
		$\OR_{M}(2) \ge 2$. 
		We will show that 
		for any order $T$ on $M$ and any $\varepsilon>0$ 
		there exist points $x_1$, $x_2$, $x_3$ $\in X$ such that $x_1 <_{T} x_2 <_{T} x_3$ and $d(x_1, x_2), d(x_2, x_3) \ge R- \varepsilon$, i.e. $(x_1,x_2,x_3)$ is a snake of large elongation.
		
		\noindent Take $\varepsilon = R/n$, $n \in \mathbb{N}$. 
		Consider points $y_1, y_2, \dots y_{2n}$ with $d(y_i, y_{i+1}) = \varepsilon$, numbered with respect to the  order $T_0$.
		Consider a map $\phi$ from $M$ to the two point set $\{0,1\}$ which takes value
		$0$ at a point $x\in M$ if the antipodal point to $x$ is $T$-smaller than $x$ and $1$ otherwise. Clearly, the preimages of $0$ and $1$ are non-empty. 
		Therefore, there exists an index $i$ (considered  modulo $2n$) such that $\phi(y_i) = 1$, 
		$\phi(y_{i+1}) = 0$.

		\[\begin{tikzpicture}[scale = 0.6, every node/.style={scale=0.6}]
		\tikzmath{\R = 3; \x1 = \R*cos(100); \y1 =\R * sin(100);
			\x2 = 3*cos(60); \y2 =3 * sin(60);} 
		\draw[thick] (0, 0) circle [radius = 3cm];
		\draw[thick, ->] (-0.95 *\x1,-0.95*\y1) -- (0.95*\x1,0.95*\y1);
		\draw[thick, <-] (0.95*\x1,-0.95*\y1) -- (-0.95*\x1,0.95*\y1);
		\draw[dashed, thick] (-0.95*\x1, 0.95*\y1) -- (-0.95*\x1, -0.95*\y1);
		\node[above] at (\x1,\y1) {$y_i,1$};
		\node[above] at (-\x1,\y1) {$y_{i+1},0$};
		\node[below] at (-\x1,-\y1) {$y_{i+N}$};
		\node[below] at (\x1,-\y1) {$y_{i+1+N}$};
		\node[right] at (\x2, \y2) {$y_{i+2}$};
		\node[left] at (-\x2, \y2) {$y_{i-1}$};
		\draw (\x1, \y1) circle[radius = 0.075];
		\draw (-\x1, -\y1) circle[radius = 0.075];
		\draw (-\x1, \y1) circle[radius = 0.075];
		\draw (\x1, -\y1) circle[radius = 0.075];
		\draw (\x2, \y2) circle[radius = 0.075];
		\draw (-\x2, \y2) circle[radius = 0.075];
		\draw (\x2, -\y2) circle[radius = 0.075];
		\draw (-\x2, -\y2) circle[radius = 0.075];
		\end{tikzpicture}\]

		This means that $y_i >_T y_{i+n}$ and $y_{i+1} <_T y_{i+1+n}$.
		If $y_{i+1} >_T y_{i+n}$, then the snake 
		$y_{i+n} <_T y_{i+1} < y_{i+n+1}$ has large elongation.
		Otherwise we take the snake $y_{i+1} <_T y_{i+n} <_T y_i$.  
	\end{proof}

	\noindent 
	This lemma has a generalisation to spheres of dimension $d$. For statements and applications we refer to Section $6$
	in \cite{ErschlerMitrofanov2}.
	
	\medskip
	
	\noindent A {\it tripod} is a metric space that consists of three segments meeting at one vertex. 
	
	\begin{corol}\label{cor:tripod}
		Let $M$ be a tripod. 
		Then $\OR_M(2)=2$. 
		If $l$ is the minimal length of the three 
		segments of $M$, then for all $\varepsilon > 0$ and for any order $T$ of $M$ there exists a snake on three points, of diameter at least $l- \varepsilon$ and of width at most $2 \varepsilon$.
	\end{corol}
	
	\begin{proof}
		We can assume that all segments of the tripod $M$ have length $1$. 	
		Consider  $N = S^1$ to be a circle of length $6$.
		There is a $1$-Lipschitz map $\varphi: N \to M$ (see the figure below) such that any two antipodal points map to a pair of points at distance  $1$.

		\begin{tikzpicture}[scale = 0.6, every node/.style={scale=0.8}]
		\draw (0,1) circle(3);
		\foreach \x in {0,...,5}{
			\node at ({3.2*cos(60 * \x + 30)}, {3.2*sin(60 * \x + 30)+1}) {$A_{\x}$};
		}
		
		\draw[->][thick] (6.5,1.1) -- (8.5,1.1);
		
		\begin{scope}[shift = {(13,0)}]
		\draw[ultra thick] (0,0) -- (0,3);
		\draw[ultra thick] (0,0) -- (2.5,-1);
		\draw[ultra thick] (0,0) -- (-2.5,-1);
		\draw (-0.2,3.2) node[above]{$A_1$} -- (-0.2, 0.2) node[above left]{$A_2$} --
		(-2.6, -0.8) to[out = -130, in = 190] (-2.4, -1.2) node[below left] {$A_3$} -- 
		(0, -0.25)node[below] {$A_4$} -- (2.4, -1.2) node[below right] {$A_5$}  to [out = -20, in = -10] (2.6, -0.8)
		--(0.2, 0.2) node[above right] {$A_0$} -- (0.2, 3.2) to [out = 90, in = 90] cycle;
		\end{scope}
		
		\end{tikzpicture}
		
		Let $T$ be an arbitrary order on $M$ and let $T_N$ be some pullback of $T$ to $N$. 
		From Lemma \ref{le:examplecircle} we know that for any $\varepsilon > 0$ there are two antipodal points $x_1, x_2 \in N$ and points $y_1 <_{T_N} y_2 <_{T_N} y_3$ such that $y_1, y_3 \in B(x_1,\varepsilon)$, $y_2 \in B(x_2,\varepsilon)$. Here  $B(x, \varepsilon)$ denotes an open ball of radius $\varepsilon$ centered at $x$.
		
		We know that $d_{M}(\varphi(A_1), \varphi(A_2)) = 1$.
		Then $d_M(\varphi(y_1),\varphi(y_2))$ and $d_M(\varphi(y_1),\varphi(y_2)) > 1 - 2\varepsilon$, and 
		$d_M(\varphi(y_1),\varphi(y_3)) < 2\varepsilon$.
		Since $T_N$ is a pullback of $T$ we can conclude that $\varphi(y_1) <_{T}\varphi(y_2) <_{T}\varphi(y_3)$.
		This implies the second claim of the Corollary, and in particular that
		$\OR_M(2) = 2$.
	\end{proof}
	
	The following example shows that for some metric spaces the choice of an order that optimizes $\OR(k)$ depends
	on $k$.
	
	\begin{example}
		[Dependence of an optimal order on the number of points]\label{rem:numberofpoints}
		Let $M$ be a six point metric space, with the distance function described by the matrix

		\[ 
		\left( \begin{array}{ccccccc}
		0 & 1 & 1.5 & 1.7& 1.5 & 2 \\ 
		1& 0& 1.8& 1.6& 1.5& 1.6  \\ 
		1.5& 1.8& 0& 1& 1.7& 2  \\ 
		1.7 & 1.6& 1& 0& 1.3& 1.6  \\ 
		1.5& 1.5& 1.7& 1.3& 0& 1.7  \\ 
		2& 1.6& 2& 1.6& 1.7& 0 
		\end{array} 
		\right) \] 
		
		\noindent Observe that the pairwise distances take values between $1$ and $2$, and hence satisfy the triangular inequality.
		
		\noindent It can be checked that the only orders that minimize $\OR(2)$ are 
		$T_1 = (1,2,3,4,5,6)$ and $T_2 = (6,5,4,3,2,1)$, but these two orders are not optimal for $k=3$.
		It can be shown that the order $T_3 = (3, 4, 5, 1, 2, 6)$ satisfies $\OR_{M, T_3}(3) < \OR_{M, T_1}(3) = \OR_{M, T_2}(3)$.
	\end{example}

	\subsection{Orders on trees. Hierarchical orders}\label{subsec:trees}
	
	Consider a finite directed rooted tree $\Gamma$.
	The root of the tree is denoted by $\mathcal{O}_{\Gamma}$. Any vertex  can be joined with the root by a unique  directed path, and in particular each vertex has at most one
	parent. We assume that the direction on each edge is from the parent to its  child.
	
	We also assign to  each edge  some positive length, 
	and this
	provides a structure of a metric space on
	the vertices of $\Gamma$. 
	We say that $x$ is an {\it ascendant} of $y$ if there exists a sequence of vertices $(x_0 =x, x_1, \dots , x_l=y)$, such that for each $i$ the vertex $x_i$ is a parent of $x_{i+1}$.
	
	Let $T_{\rm rt}$ be an order on the vertices of the rooted tree $\Gamma$, defined as follows.
	First for any vertex we fix an arbitrary order on the children of this vertex.
	Now if a vertex
	$x$ is an ascendant of $y$ then $x<_{T_{\rm rt}} y$. Otherwise consider paths of minimal length from $\mathcal{O}_{\Gamma}$ to $x$ and $y$: $(v_0=\mathcal{O}_{\gamma}, v_1, \dots, v_k=x)$ and $(u_0=\mathcal{O}_{\Gamma} , u_1, \dots, u_m=y)$.
	Observe that  if neither $x$ is an ascendant of $y$ nor $y$ is an ascendant
	of  $x$, then there exists $i: i \le k, i\le m$ such that $v_i \ne u_i$. Take minimal $i$ with this property and put
	$x>_{T_{\rm rt}}y$ if $v_i> u_i$ (in our fixed order on the children of $v_{i-1}= u_{i-1}$).
	
	Any order obtained in this manner we call \textit{a rooted order}.
	Let us say that a subset of a rooted tree is a {\it branch} if it consists of some vertex $x$ and all vertices $y$ such that $x$ is an ascendant of $y$.
	We denote this  subset  $\Gamma_x$.
	
	Let us say that an order $T$ on vertices of a rooted tree is {\it hierarchical} if the following holds. 
	Suppose that
	$x<_T y<_T z$ and $x$ and $z$ belong to some branch. Then $y$ also belongs to this branch.
	
	The existence of hierarchical orders  can be deduced from Lemma \ref{le:convex}. Without referring to this Lemma, it is also not difficult to see it more directly, in Lemma
	\ref{lem:hierarchicaltrees}.
	Claim (1) of this Lemma is proven
	in Thm 2, \cite{Schalekamp}.
	For the convenience of the reader we provide the proof.
	
	\begin{lemma} \label{lem:hierarchicaltrees}
		Let $\Gamma$ be a rooted tree as above, then
		\begin{enumerate} 
			\item 
			For any  $k\ge 1$ and any rooted order $T_{\rm rt}$ 
			it holds that
			$$
			\OR_{V(\Gamma), T_{\rm rt}} (k) \le 2,
			$$
			where $V(\Gamma)$ is the set of vertices of $\Gamma$.
			
			\item Any rooted order $T_{\rm rt}$ is hierarchical.
			
		\end{enumerate}
	\end{lemma}
	
	\begin{figure}[!htb] 
		\centering
		\includegraphics[scale=.45]{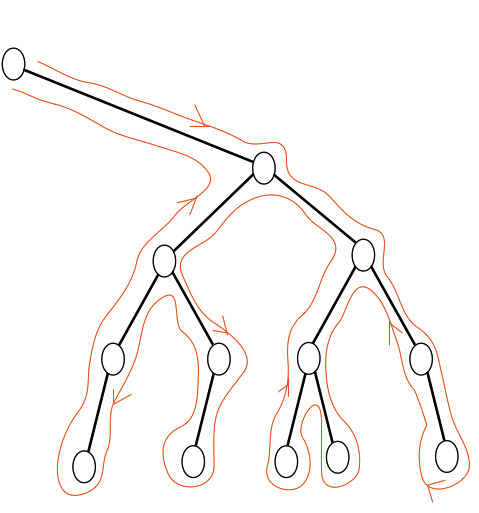}
		\caption{Vertices of  the tree are ordered with respect to  the first visit by the orange path. 
			In case when the tree is a Cayley graph of a free group $F_n$, one of such orders corresponds
			to a lexicographic order of $F_n$.}
		\label{pic:orderedtree}
	\end{figure}
	
	\begin{proof} Let $S$ be a  subset of
		vertices of $\Gamma$, $\# S \ge 2$. 
		Observe that any (possibly self-intersecting) continuous path  passing through $S$ passes
		through all the vertices of 
		the spanning tree $\Gamma'$ of $S$. 
		Observe also that $l_{\opt}(S)$ is at least the sum of the lengths of all edges in $\Gamma'$. 
		It is clear that the restriction of $T_{\rm rt}$ to $\Gamma'$ is a rooted order on $\Gamma'$.
		
		Now consider a path through vertices of  $S$ with respect to our order $T_{\rm rt}$. Let us show that this path passes at most twice
		through each edge of the spanning tree, this would show that the length of this path is at most twice the length of the optimal tour. Without loss of generality we can assume that $V(\Gamma)=S$.
		We prove the claim by induction on the number of vertices of the tree $\Gamma$.
		For each child of $\mathcal{O}_{\Gamma}$, consider the corresponding branch. The tour with respect to $T_{\rm rt}$ first visits $\mathcal{O}_{\Gamma}$, then one of its children and all points of its branch, then another child and all points of its branch.
		Observe that the edge from $\mathcal{O}_{\Gamma}$ to any of its children is visited at most twice: once passing from $\mathcal{O}_{\Gamma}$ to its child, and, possibly second time, after visiting all points of the branch of this child, before passing to another child of $\mathcal{O}_{\Gamma}$. 
		Now observe that  all edges not adjacent to $\mathcal{O}_{\Gamma}$ are visited at most twice by induction hypothesis, applied to the branches of $\Gamma$.
		So we have proved the first claim of the lemma.
		
		Now let $x$ be a vertex of $\Gamma$ and $\Gamma_x$ be a branch. 
		From the argument above it follows that the tour with respect to $T_{\rm rt}$ after visiting $x$ visits all  points of $\Gamma_x$, before   visiting any other points (not in $\Gamma_x$). 
		This implies the second claim of the lemma that the order is hierarchical.
	\end{proof}

	Recall that a metric space $M$ is a {\it metric tree} if it is a geodesic
metric space which is $0$-hyperbolic.

\begin{corol}\label{re:anytree}
	Let $M$ be a  metric tree. 
	Then there is some order $T$ such that $\OR_{M,T}(k) \le 2$ for all $k$.
\end{corol}
\begin{proof}
	By Lemma \ref{rem:goedel} we know that if the statement that $\OR(k) \le 2$ is true for any finite tree, then it is true for any metric tree.
\end{proof}

We say that an order $T$ on a metric tree $M$ is \emph{ hierarchical} if for any finite subset $S\subset M$
there is a choice of root $x\in S$ (which includes a rooted tree structure on $S$) such that the restriction of $T$ to $S$ is hierarchical.
Hierarchical orders always exist. To show this we can choose a point $x\in M$ as the root and obtain an order $T$ by applying Lemma \ref{le:convex} to the collection $\mathcal{A}$ that consists of all subsets $A \subset M$ such that $A$ is a component of $M\setminus \{y\}$ for some $y \in M$ and $x \not \in A$.
Given a finite subset $S \subset M$, we can choose the root of $S$ to be any closest point to $x$. It is then not hard to show that the restriction of $T$ to $S$ is hierarchical.
It can be shown that any hierarchical order satisfies the conclusion of Corollary \ref{re:anytree}, the proof is essentially the same as the proof of the first statement of Lemma \ref{lem:hierarchicaltrees}.
	
	\begin{remark}\label{rem:freegroup}
		Let $G$ be  a free group with a free generating set $S$, $\#S > 1$. 
		If we consider $G$ as a metric space with the word metric with respect to $S$, 
		then $G$ embeds isometrically into its Cayley graph, which is a metric tree.
		Corollary \ref{re:anytree} implies that $\OR_G(k) \leq 2$ for all $k\ge 3$.
		Now we describe a more explicit order. 
		Recall that each element of $G$ corresponds to a unique irreducible word in the alphabet $S\cup S^{-1}$. Choose an order on the set $S\cup S^{-1}$, then the lexicographic order $T_{lex}$ on words on this alphabet provides an example of a hierarchical order on $G$ with $\OR_{G,T_{lex}}(k) \leq 2$ for all $k$.
	\end{remark}

	\subsection{Acyclic gluing of spaces}
	
	Let us say that a metric space $M$
	is an {\it acyclic gluing} of metric spaces $M_\alpha, \alpha \in \mathcal{A}$ 
	if it is obtained as follows.
	Consider a graph which is a tree (finite or infinite), vertices of which have labellings of two types: 
	each vertex is either labelled by 
	$M_\alpha$, for some  $\alpha \in \mathcal{A} $ or it is labelled by a one-point space $\{x\}$.
	Each edge $e$ joins some $M_{\alpha}$ and $\{x\}$, and is labelled by a point $f_e(x)\in M_{\alpha}$. 
	
	Our space is obtained from a disjoint union of $M_{\alpha}$ and points $x$ by gluing all $x$ to the points $f_e(x)$ in corresponding $M_\alpha$.
	
	A graph which is an acyclic gluing of circles and intervals  is also called 
	{\it a cactus graph}.
	An example of a space homeomorphic to a cactus graph, with only circles glued together, 
	is shown on Figure \ref{pic:wedgescircles}.
	Observe, that a finite acyclic gluing can be obtained by a finite number of operations which take
	a wedge sum of two metric spaces.
	
	\noindent We will describe a natural way to order points of $M$ if orders on $M_{\alpha}$ are given. 
	
	\noindent For acyclic gluings it will be more natural to consider {\it cyclic order ratio functions}.
	
	\begin{definition}\label{def:cyclicOrderingRatioFunction}
		Let $M$ be a metric space with an order $T$.
		If $X$ is a finite subset of $M$ we denote by $l^{\circ}_{\opt}(X)$ the minimal length of a closed path visiting all points of $X$.
		We denote by $l_T^{\circ}(X)$ the length of the closed path corresponding to the order $T$: if the distance between the first and the last vertices of $X$ is equal to $r$, then  $l_T^{\circ}(X) = l_T(X) + r$.
		
		\noindent We define the {\it cyclic order ratio function} as
		$$
		\OR_{M,T}^{\circ}(k) := \sup_{X\subset M : 2 \le \#X \le k+1} \frac{l^{\circ}_T(X)}{l^{\circ}_{\opt}(X)}.
		$$

	\end{definition}
	
	\begin{lemma}\label{le:snakecycle} For any ordered metric space $(M,d,T)$ and any $k$ it holds that
		
		$$
		1 \le \OR^{\circ}_{M,T}(k) \le \frac{k+1}{2}.
		$$
		Moreover if $k$ is odd, then the following conditions are equivalent:
		\begin{enumerate}
			\item $\OR^{\circ}_{M,T}(k) = (k+1)/2$
			\item There exist snakes on  $k+1$ points in $(M,T)$ of arbitrarily large elongation
		\end{enumerate}
	\end{lemma}

	\begin{proof}
		The inequality $1 \le \OR^{\circ}_{M,T}(k)$ is obvious. 
		Let $X \subset M$, $2 \leq \#X \leq k+1$, and let $a$ be the diameter of $X$. 
		Then $l^{\circ}_{\opt}(X) \ge 2a$ and $l^{\circ}_{T}(X) \le (k+1)a$.
		This implies the inequality $\OR^{\circ}_{M,T}(k) \le \frac{k+1}{2}$.
		
		Let $k$ be odd and let $S = \{x_0,\dots,x_{k}\}$ be a snake of diameter $a$ and width $b$. 
		By the triangle inequality, for any $i$ we have $d(x_i,x_{i+1}) \ge a-2b$ and hence  $l^{\circ}_T(S)\ge (k+1)(a-2b)$. 
		It also holds that $l^{\circ}_{\opt}(S)\le 2a+(k-1)b$. 
		If the snake $S$ has large elongation then $l^{\circ}_T(S)/l^{\circ}_{\opt}(S)$ is close to $(k+1)/2$.
		
		Now assume that a subset $X = \{x_0,\dots, x_{k}\}$ has diameter $a$ and $\frac{l^{\circ}_T(X)}{l^{\circ}_{\opt}(X)}$ is close to $(k+1)/2$.
		Suppose that $x_0 <_T \dots <_T x_k$.
		The following argument is analogous to that in the proof of Lemma \ref{lem:snake}.
		It is easy to see that  for any $1 \le i \le k$ the distances $d(x_{i-1},x_i)$ and $d(x_{i},x_{i+1})$ are close to $a$, and $d(x_{i-1},x_{i+1})$ is close to 0 (we assume here that $x_{k+1} = x_0$). 
		We can deduce that the width of the snake $X$ is also close to 0, and $X$ has large elongation. 
	\end{proof}
	
	The equivalence of $(1)$ and $(2)$ in the lemma above means that for odd $k$, $\OR_{M,T}(k) = k$ if and only if $\OR^{\circ}_{M,T}(k)= (k+1)/2$.		
	
	\begin{definition}
		Let $M$ be a metric space (or more generally a set) and $T$ be an order on $M$. 
		Suppose that $M$ is a disjoint union $M =  A \bigsqcup B$ and that for any two points $x\in A$, $y \in B$ we have $x <_T y$.
		We define a new order $T'$ on $M$ as follows: if $x, y \in A$ or $x, y \in B$ we put $x <_{T'} y$ if $x <_T y$. 
		For any two points $x \in A$ and $y\in B$ we put $x >_T y$.
		We say that $T'$ is a {\it cyclic shift} of the order $T$.
	\end{definition}
	
	It is clear that for a given metric space the relation to be a cyclic shift is an equivalence relation on orders.
	For example, all clockwise orders of a circle with different starting points are cyclic shifts of each other.
	For a given metric space $M$, 
	an order $T$ and a point $x\in M$,
	there exists a unique cyclic shift $T^x$ such that $x$ is the minimal point for this order. 
	
	In the following two remarks we state straightforward properties of cyclic shifts and cyclic order ratio functions.

	\begin{rem}\label{rem:ORcyclicshift1}
		If $M$ is a metric space endowed with an order $T$, then for any finite subset $X$ of $M$ it holds that
		
		$$
		1 \leq \frac{l^{\circ}_T(X)}{l_T(X)} \leq 2; \:\:\:\: 
		1 \leq \frac{l^{\circ}_{\opt}(X)}{l_{\opt}(X)} \leq 2.
		$$
		
		It follows that  for any $k$ 
		$$
		\frac{1}{2} \OR_{M,T}(k) \leq \OR^{\circ}_{M,T}(k) \le 2 \OR_{M,T}(k).
		$$
		
		For any $m$, $\mins(M,T) \geq 2m$ if and only if $\OR^{\circ}_{M,T}(2m-1) = m$.
	\end{rem}
	
	\begin{rem}\label{rem:ORcyclicshfift2}
		If $T'$ is a cyclic shift of $T$, then for any $k$ we have
		$$
		\OR^{\circ}_{M,T'}(k) = \OR^{\circ}_{M,T}(k),
		\: 
		\OR_{M,T'}(k) \le 2 \OR_{M,T}(k)
		$$
		
		In particular, taking in account the last claim of the previous remark 
		$\OR_{M,T}(2m-1) = 2m-1$ if and only if 
		$\OR_{M,T'}(2m-1) = 2m-1$.
	\end{rem}

	Now given a family $M_\alpha$ and a family of orders
	on these spaces $T_\alpha$, we consider a
	metric space $M$ 
	which is an acyclic gluing of ordered spaces $M_{\alpha}$, $\alpha \in \mathcal{A}$. We define  a family of  {\it  clockwise orders} on $M$ as follows.
	We assume below that the tree in the definition of acyclic gluing is finite (this is assumption is not essential).
	One way  is to define such orders recursively:
	having already defined  orders $T_A$ and $T_B$ on two metric spaces $A$ and $B$, we will define an order $T_C$ on the wedge sum $C$ of $A$ and $B$ over a point $x$.
	We 
	choose any of the two ways to enumerate $A$ and $B$, for example $(A, T_A)=(M_1, T_1)$ and $(B, T_B)=(M_2, T_2)$. 
	We use the notation that denotes by $x$ both a point in $M_1$ and one in $M_2$.
	
	Recall that for any point in any space, there is a cyclic shift that makes this point minimal.
	Let $T^x_1$ and $T^x_2$ be  cyclic shifts of $T_1$ and $T_2$ respectively such that $x$ is the minimal point in $(M_1,T^x_1) $ and $(M_2, T^x_2)$. 
	Consider the order $T_C$ on a wedge sum of $M_1$ and $M_2$, such that $x$ is a minimal  point, then come all points of $M_1 \setminus \{ x \}$ ordered as in  $T^x_1$ and
	then all points of $M_2 \setminus \{ x \}$
	ordered as in  $T^x_2$.
	In such a way  we construct recursively  an order on an acyclic gluing. 
	We say that such orders and any cyclic shift of these orders are  {\it clockwise orders} on $M$.
	Finally, if the tree in the definition of an acyclic gluing is infinite, we can say that an order on this space is a clockwise order if the restriction of this order to any finite (connected) acyclic gluing
	is a clockwise order.
	
	There is another way to formulate this definition and to visualize clockwise orders. 
	Given an acyclic gluing $M$ of $M_\alpha$, $\alpha \in \mathcal{A}$, we consider circles $N_\alpha$, $\alpha \in A$ and choose a direction (clockwise order) on each $N_{\alpha}$.
	For all joint points on $M_\alpha$ we consider points on $N_\alpha$ enumerated in the same order, up to a cyclic shift, and we consider the corresponding acyclic gluing of $N_\alpha$. 
	The clockwise order on an acyclic gluing $N$ of $N_\alpha$ can be obtained in the following way.
	We construct a continuous embedding of $N$ in the plane in such a way that 
	the image of $N$ is equal to its outer boundary, and the clockwise orientation of the boundary  of the image of $N$  (in the plane) coincides with clockwise orientations of $N_{\alpha}$.
	
	We choose a joint point $O$ (among joint points joining our circles)
	as a base point  and  consider an order corresponding to the first visit of a clockwise path in the plane, see
	Figure \ref{pic:wedgescircles}.
	For each point of $M_\alpha$ we can associate an arc in $N_\alpha$. If $m$ is a joint point in $M$, then it corresponds
	to a joint point in $N$. We will use a convention that the initial point of an arc (with respect to the first visit of the clockwise path starting from $O$) is included in the arc and the last point of the arc is not included.
	
	A clockwise order on $M$ can be described as follows. Take $m, m' \in M$. First suppose
	that there exists $\alpha$ such that $m, m' \in M_\alpha$. Observe that our choice of $O$ fixes a
	choice of a cyclic shift on the acyclic gluing of our circles (and in particular on  $N_\alpha$). This induces a choice of an acyclic shift also 
	on $M_\alpha$.
	In case when there exists $\alpha \in \mathcal{A}$ such that $m, m' \in M_\alpha$ we compare $m, m'$ with respect to the above mentioned cyclic shift of $T_\alpha$ (on $M_\alpha$).
	Finally, if there is no $\alpha$ such that $m, m'$ belong to the same $M_\alpha$, we consider corresponding arcs $\gamma$, $\gamma'$, choose any point $n\in \gamma$, $n' \in \gamma'$ and compare
	$n$ and $n'$ in $N$. If $n<n'$ with respect the order in $N$,
	then we say that $m<m'$
	
	\begin{figure}[!htb] 
		\centering
		\includegraphics[scale=.45]{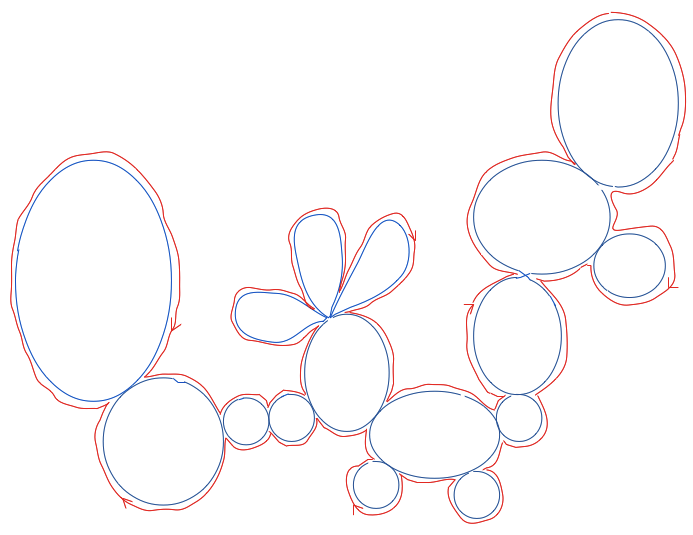}
		\caption{An acyclic gluing of circles. A clockwise order corresponds to a first visit by an orange path. More precisely,
			to choose one among cyclic shifts of this clockwise order, we fix any point of the gluing and consider the first visit by an orange path starting at this point.}
		\label{pic:wedgescircles}
	\end{figure}
	
	Now observe that a metric tree is an acyclic gluing of intervals, and recall 
	that $\OR(k) \le 2$ for all $k$ in this case. 
	Here we prove a statement
	about general acyclic gluings.
	
	\begin{lemma} \label{le:ORacyclicunion}
		Let $M$ be an acyclic gluing of ordered spaces $(M_\alpha, T_{\alpha})$, $\alpha \in \mathcal{A}$. 
		Let $T$ be a clockwise order on $M$. 
		For all $k$ we have
		$$
		\OR^{\circ}_{M,T}(k) =  \sup_\alpha \OR^{\circ}_{M_\alpha, T_\alpha} (k).
		$$
	\end{lemma}
	
	\noindent If we allow metric spaces with infinite
	distances between some points, there is a way to define $\OR$ also in this more general setting (considering subsets $S$ with finite pairwise distances)
	and in this setting the lemma can be reformulated to claim that order ratio function of an acyclic gluing $M$ is the same as for the disjoint union of $M_\alpha$.

	\begin{proof}
		It is obvious that the left hand side in the equation in  the lemma  is at least
		as large as the right hand side, so it is enough to prove that
		$$
		\OR^{\circ}_{M,T}(k) \leq  \sup_\alpha \OR^{\circ}_{M_\alpha, T_\alpha} (k).
		$$
		
		To prove this inequality, 
		we first observe that it is enough to
		prove this claim for finite acyclic gluings.
		We consider a finite set $X$ consisting of $k+1$ points $x_1, x_2, \dots x_{k+1}$.
		Consider a minimal (with respect to inclusion) acyclic gluing $M'$ of $M_\alpha$, 
		$\alpha \in \mathcal{A'} \subset \mathcal{A}$, containing $X$.
		If $M_\alpha$ are path-connected metric spaces,
		then observe that any continuous path passing through $X$ visits each $M_\alpha$,
		for $\alpha \in \mathcal{A}'$. 
		Even though $M_\alpha$ are not assumed to be path-connected in general, for any two points 
		$x, y$ we can consider a finite sequence of points $x=z_1$, $z_2$ \dots $z_p=y$, where $p$ depends on $x$ and $y$, such that 
		for any $i$ the points $z_i$ and $z_{i+1}$ belong to the same component $M_{\alpha}$, and
		the distance between $x$ and $y$ is
		$$
		\sum_{i: 1\le i \le p-1} d_M(z_i, z_{i+1}).
		$$
		It is clear that such $z_i$ exist and we can choose the $z_i$ such that for $1<i <p$ these points are joint points in the acyclic gluing. 
		
		Now we choose a component $M_\alpha$, $\alpha \in \mathcal{A}'$ and consider 
		the retract mapping $\rho_{\alpha}: M\to M_{\alpha}$ that maps every point of $M$ to the nearest point of $M_{\alpha}$. 
		It is clear that if $x\not\in M_{\alpha}$ then $\rho(x)$ is a joint point.
		Observe, that the retract $\rho_{\alpha}$  sends the clockwise order on $M$ to some (cyclic shift) of $T_\alpha$.
		More precisely, if we choose an arbitrary point $y\in M_{\alpha}$ that is not a joint point and consider the cyclic shift $T^{y}$ of $T$, then
		for any two points $z_1 <_{T^y}z_2$ we have
		$\rho_{\alpha}(z_1) \leq_{T^y} \rho_{\alpha}(z_2)$.
		
		\begin{figure}[!htb] 
			\centering
			\includegraphics[scale=.45]{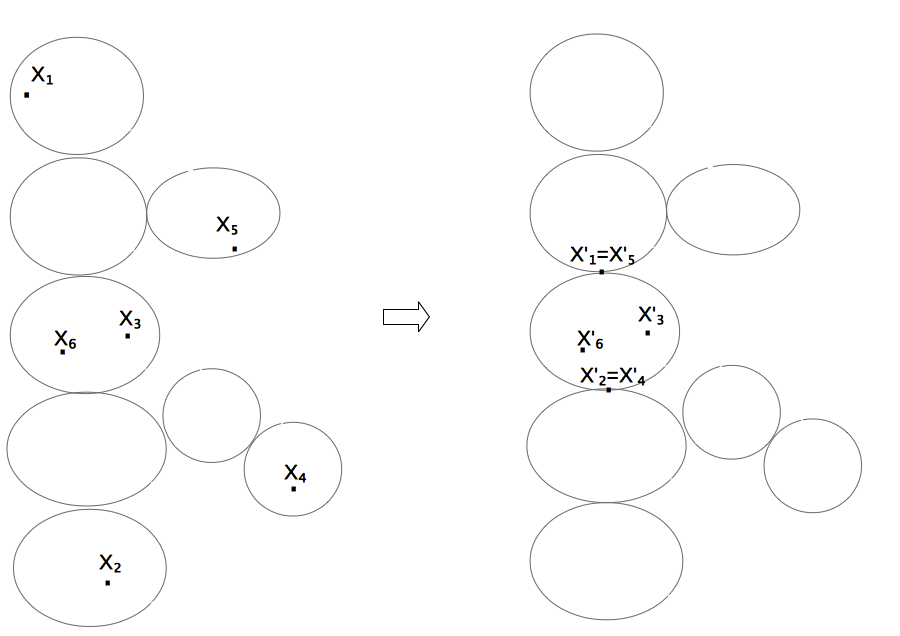}
			\caption{A retract of the gluing to s component 
				and the images of $x_i$.}
			\label{pic:wedgessumOR}
		\end{figure}

		Observe also that any closed path
		passing through  $x_i$ ($1 \le i \le k+1$) passes through all their images $\rho_{\alpha}(x_i)$ ($1 \le i \le k+1$).
		To be more precise, if $M_\alpha$ are not necessarily geodesic, we use the same convention as before: 
		by a path in $M$ we mean a sequence of points such that any two consecutive points are in the same component. 
		We call such pairs of consecutive points {\it jumps}.
		
		For the shortest closed path visiting all points of $X$ we consider all its jumps
		inside $M_\alpha$.
		These jumps form a cyclic tour visiting all the points of $\rho_{\alpha}(X)$, and their total length is at least $l^{\circ}_{\opt}(\rho_{\alpha}(X))$.
		Now we consider a tour of $X$ with respect to a clockwise order $T$, and take all its ``jumps'' inside  $M_{\alpha}$.
		They form a tour of $\rho_{\alpha}(X)$ with respect to $T$ and $T_{\alpha}$, and its length can be bounded from above as follows
		
		$$
		l^{\circ}_{T}(\rho_{\alpha}(X)) 
		\leq \OR^{\circ}_{M_{\alpha}, T_{\alpha}}\left(\#\rho_{\alpha}(X)\right)
		l^{\circ}_{\opt}(\rho_{\alpha}(X)) \leq
		\OR^{\circ}_{M_{\alpha}, T_{\alpha}}(k)
		l^{\circ}_{\opt}(\rho_{\alpha}(X)).
		$$
		
		Summing these inequalities over all $\alpha\in \mathcal{A}'$ and using the obvious inequality 
		
		$$
		\sum_{\alpha\in \mathcal{A}'}l^{\circ}_{\opt}(\rho_{\alpha}(X)) \le l^{\circ}_{\opt}(X),
		$$
		we prove the claim of the lemma.
	\end{proof}
	
	In Remark \ref{rem:ORcyclicshift1} we mention that $\OR$ and $\OR^{\circ}$ are the same up to the multiplicative constant $2$.
	So we get the following
	
	\begin{corol}
		Let $M$ be an acyclic gluing of ordered spaces $(M_\alpha, T_{\alpha})$, $\alpha \in \mathcal{A}$. 
		Let $T$ be a clockwise order on $M$. 
		Then for any $k \ge 1$
		$$
		\OR_{M,T}(k) \le 4 \sup_\alpha \OR_{M_\alpha, T_\alpha} (k).
		$$
	\end{corol}
	
	As another  corollary  we obtain 
	\begin{corol}\label{cor:acyclicgirth}
		
		Let $M$ be an acyclic gluing of ordered spaces $(M_{\alpha}, T_{\alpha})$, $\alpha \in \mathcal{A}$, and let $T$ be a clockwise order on $M$. 
		For a finite acyclic gluing $M$ 
		$\OR_{M,T}(2k-1)=2k-1$ if and only if
		for some $\alpha$ it holds that $\OR_{M_\alpha,T}(2k-1)=2k-1$.
		More generally, for a not necessarily finite acyclic gluing we have
		$\OR_{M,T}(2k-1)=2k-1$ if and only if 
		there is no uniform upper bound for all $M_{\alpha}$ on elongations of snakes on $2k$ points in $M_{\alpha}$.
	\end{corol}
	
	\begin{proof}
		The ``if'' directions obviously follow from Lemma \ref{lem:snake}.
		Now suppose that for each $\alpha$ the elongations of snakes on $2k$ points in all $(M_{\alpha},T_{\alpha})$ are uniformly bounded from above.
		Using the upper bound for $\OR^{\circ}(2k-1)$ in terms of the elongations of snakes (see the proof of Lemma \ref{le:snakecycle}), we observe that there exists $r < k$ such that 
		$\OR^{\circ}_{M_{\alpha}, T_{\alpha}}(2k-1) \le r$ for any $\alpha$.
		Lemma \ref{le:ORacyclicunion} shows that $\OR^{\circ}_{M,T}(2k-1)\le r < k$, and the result follows from Lemmas \ref{le:snakecycle} and \ref{lem:snake}.
	\end{proof}
	
	The following corollary is about $\mins$ for free product of groups. 
	In the particular case of virtually free groups we will show later that they can be characterized as groups with small $\mins$.
	In this corollary the metric associated to the groups is the word metric on the elements of these groups (we consider only vertices of  Cayley graphs, not edges).
	\begin{corol} Let $G=A*B$ be a free product of groups $A$ and $B$. 
		If the maximum of $\mins(A)$, $\mins(B)$ is odd, then $\mins(G)$ is equal to this maximum.
		If the maximum is even, then $\mins(G)$ is either equal to this maximum or the maximum plus one.
	\end{corol}
	
	\begin{proof}
		Observe that the metric space associated to $G$ (that is, the word metric on this free product) is an acyclic gluing of metric spaces $M_{\alpha}$ where each $M_{\alpha}$ is isometric to $A$ or to $B$.
		It is clear that $\mins(G) \ge \max (\mins(A), \mins(B))$.
		Let $k$ be an integer such that $\max (\mins(A), \mins(B)) \leq 2k-1$, then $\OR_A(2k-1) < 2k-1$ and $\OR_B(2k-1) < 2k-1$. 
		From Corollary \ref{cor:acyclicgirth} it follows that $\OR_G(2k-1) < 2k-1$ and $\mins(G) \le 2k-1$.
	\end{proof}
	
	If we do not assume that the maximum is odd,  the ``plus one'' in the formulation is essential.
	For example, $\mins(\mathbb{Z}) = 2$ and $\mins(\mathbb{Z}*\mathbb{Z})=3$ (see Lemma \ref{le:notrayline}).

	\subsection{Order breakpoint of finite graphs.}
	
	In this subsection we consider both finite and infinite graphs, but the main goal is to classify finite graphs depending on their order breakpoints.
	Let $\Gamma$ be a graph, finite or infinite, with edges of length $1$. 
	We consider $\Gamma$ as a geodesic metric space with edges included. 
	Given an order on the set of vertices of $\Gamma$, we define an order on all points of $\Gamma$ as follows.
	
	\begin{definition}
		\label{def:star}
		Given an order $T$ on the set of vertices $V$ of $\Gamma$, 
		we define the order ${\rm Star}(T)$ on $\Gamma$. We subdivide each edge $AB$ into two halves of length $1/2$, and we assume
		that the middle point belongs to the half of $A$ if $A<_T B$.
		We call a union of a vertex and all the outgoing half-edges a {\it star figure}.			
		For each point $x$ in $\Gamma$ we have defined therefore the vertex $v(x)$ (to the half-interval of which it belongs).
		If $v(x)<_T v(y)$ we put $x <_{{\rm Star}(T)} y$. On points with the 
		same $v(x)$ we consider a hierarchical order on the corresponding star figure, assuming that the vertex is smaller than the other points. 
		So we have an order on the set of half-edges,
		and each half-edge is ordered by identifying it with the interval $[0,1]$ or $[0,1)$, where 0 corresponds to the vertex.
	\end{definition}
	A comment on the definition above: the choice to which half edge the middle point belongs, and the order  we have chosen for half edges with a common vertex (the corresponding star figure) is not important.
	We have made this choice to fix an order on star figures which we will study in the sequel.
	
	Lemma \ref{lem:edgesnotedges} below shows that for  the asymptotic behaviour of the order ratio function it does not matter whether we consider the graph together with the edges or only the metric
	on the vertices. In a more general setting,  the asymptotic behavior of the order ratio function 
	depends on both the local and global geometry of the metric space.

	\begin{lemma}\label{lem:edgesnotedges}
		For any graph $\Gamma$ and all
		$k\ge 2$  it holds that
		$$
		\OR_{\Gamma, \Star(T)}(k) \leq 8 \OR_{V,T}(k) + 4.
		$$
	\end{lemma}
	This lemma is proven by  associating to a path in $(\Gamma, \Star(T))$ a path that passes through the centers of corresponding star figures. For details of the proof see Appendix \ref{app:edgesnotedges}.
	We do not use this lemma in the sequel, but we will use some other properties of the $\Star$ orders.

	It is clear that order breakpoint can change
	by applying  ${\rm Star}$ operation: for example, $\mins$ of the tripod is $3$ while $\mins$ of the set of its four vertices is $2$. However, this operation provides a bound on $\mins$ of the graph.
	
	\begin{lemma}\label{lem:stargirth}
		If $\Gamma$  is any (finite or infinite)
		graph, then $\mins(\Gamma, {\rm Star}(T))$ is at most $\max (\mins({V, T}), 4)$.
		In particular, if  $\Gamma$ is a finite graph, then $\mins(\Gamma, {\rm Star}(T)) \le 4$.
	\end{lemma}
	
	\begin{proof}
		Let $C$ be a constant such that any star figure does not
		admit snakes on $4$ points of elongation greater
		than $C$ with respect to the order ${\rm Star}(T)$ (from the proof of Lemma \ref{lem:snake} it follows that for any 
		hierarchical order on a metric tree, the elongation of snakes on $4$ points is bounded by $C = 10$). 
		Let $k\ge 5$ be such that
		$(V,T)$ does not admit snakes of arbitrarily large elongation on $k$ points.
		Let us  show that 
		$(\Gamma, {\rm Star}(T))$ does not admit snakes of arbitrarily large elongation on $k$ points.
		Suppose that this is not the case and $\Gamma$ admits such snakes.
		For each point $x_i$ of the snake $(x_i)$, $1\le i \le k$,  consider the vertex $v_i\in V$
		such that $x_i$ belongs to the star figure of $v_i$.
		If the elongation of the snake is 
		$> C$,
		then it is clear that among the vertices $v_i$ there are at least  two distinct points.
		
		Let $D$ and $\delta$ be the diameter and the width of the snake $(x_i)$, $1 \le i \le k$.
		First suppose that all $v_i$, $1\le i \le k$ are distinct.
		Observe  that $\delta \ge 1/2$. Indeed, $x_1$, $x_3$, $x_5$ belong to three distinct star figures,
		so not all the pairwise distances between these points can be less than $1/2$.
		We can assume that the elongation of the snake is $\ge 4$, hence $D \ge 4 \delta \ge 2$.
		Observe that in this case  the corresponding vertices $v_i$ of $\Gamma$, $1 \le i \le k$,   form a snake of  diameter $D'$ satisfying  $D'\ge D-1 \ge D/2$ and
		of width $\delta' \le \delta+1 \le 3 \delta$. We see that under our assumption the elongation of the snake $(v_i)$ 
		is at least $1/6$ of the elongation of $(x_i)$.
		
		Now observe that if there exist at least two points of the snake belonging to the same star figure $S$, then (since the star figure is convex with respect to the order on $\Gamma$) there exist two points of indices of different parity belonging to $S$.  We see that $D\le 1+2 \delta$. Assume that the elongation $D/\delta \ge 6$. We have $6 \delta  \le D \le 1+ 2\delta$, and
		$\delta \le 1/4$.
		We conclude that all points of our snake are in the ``extended'' star figure $S'$ with edges of length $1$.
		We know that not  all $x_j$ belong to $S'$, 
		and that this star figure is convex with respect to the Star order. Thus, reversing if  necessary our order we can assume that $x_1$ does not belong to this star figure.
		We know that the distance between $x_1$ and $x_3$ and between $x_3$ and $x_5$ is at most $1/4$.
		This implies in particular that these points belong to the same edge  of $\Gamma$ as $x_1$.
		
		If $x_3$ belongs to the same  half-edge as $x_1$, then,
		taking in account convexity of this half-edge, we conclude that  $x_2$ also belong to the same half-edge. 
		In this case the elongation of the 
		$3$-point snake $(x_1,x_2,x_3)$ is $1$. 
		Its diameter is $\ge D-2\delta$ and its width is $\le \delta$, where $D$ and $\delta$ are the parameters of the initial snake.
		Hence $1 \ge (D-2\delta)/\delta$, and the elongation $D/\delta$ of the initial snake is at most $3$. 
		
		Finally, if $x_3$ does not belong to the same (external) half-edge as $x_1$, 
		then $x_5$ can not belong to the same interval as $x_1$ because the half edge is convex with respect to the order. 
		Then $x_3$ and $x_5$ belong to the same half-edge.
		In this case we obtain a snake 
		$(x_3, x_4, x_5)$ on $3$ points with elongation $1$ and again get a contradiction.
		
		We have proved the first claim of the lemma. The second claim follows from the fist one, since it is clear that finite sets do not admit snakes on three points of large elongation.
		
	\end{proof}
	
	For the rest of this section, we assume that our graph is finite.
	As before, we consider a metric space of a finite graph $\Gamma$, with edges (not only vertices) included.

	Now we compute the order breakpoint for such graphs. It is clear that $\mins$ is equal to $2$ if $\Gamma$ is homeomorphic to an interval.
	We have seen already that $\mins$ is equal to $3$ for a circle or a tripod. We start with an example with $\mins(\Gamma) =4$. 	
	The {\it domino graph} is the graph with 6 vertices and 7 edges, shown at Figure \ref{pic:domino}.

	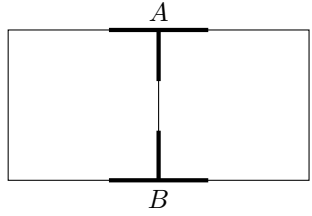
\begin{figure}[!htb] 
		\centering
		\begin{tikzpicture}[scale = 2]
		\draw (0,0) -- (0,1) -- (1,1) -- (1,0) -- (-1,0) -- (-1,1) -- (0,1);
		\draw[ultra thick] (-0.33,0) -- (0.33,0);
		\draw[ultra thick] (-0.33,1) -- (0.33,1);
		\draw[ultra thick] (0,0) -- (0,0.33);
		\draw[ultra thick] (0,1) -- (0,0.66);
		\node[above] at (0, 1) {$A$};
		\node[below] at (0, 0) {$B$};
		\end{tikzpicture}
		\caption{Two tripods in a domino graph.}
		\label{pic:domino}
	\end{figure}

	\begin{lemma}\label{lem:domino}[Order breakpoint of a domino graph]
		Let $\Gamma$ be a graph homeomorphic to the domino graph. 
		Then for any order $T$
		on $\Gamma$, this graph admits snakes on $4$ points of large elongation. More precisely, for any $\varepsilon>0$ there exists a snake on 4 points of diameter $\ge 1/3-\varepsilon$ and of width at most $2 \varepsilon$.
	\end{lemma}
	
	\begin{proof} Consider two tripods ${\rm Tr}_1$ and ${\rm Tr}_2$, with segments  of length $1/3$, with their joint vertices at $A$ and $B$ (see Figure \ref{pic:domino}).
		By Corollary \ref{cor:tripod}
		we know that in ${\rm Tr}_1$  there exists a snake on  three points of diameter at least $1/3-\varepsilon$ and of width at most $2 \varepsilon$. 
		Denote its points $x_1$, $x_2$, $x_3$, here $x_1<_T< x_2<_T x_3$.
		Without loss of generality, the distance between $x_1$ and $x_2$ is $\ge 1/3 - \varepsilon$.
		Observe that there is a continuous path from $x_2$ to any point of ${\rm Tr}_2$ which stays at distance at least $1/3-\varepsilon$ from $x_1$. If $\Gamma$ equipped with 
		our order $T$ does not admit
		snakes on $4$ points of width at most $2\varepsilon$ and diameter at least $1/3-\varepsilon$, we would see that any point on this continuous path stays (with respect to the order $T$) between $x_1$ and $x_3$. 
		This would show that all points of ${\rm Tr}_2$ are between some two points of the first tripod ${\rm Tr}_1$.
		The same argument shows that all points of the first tripod are between some two points of the second one, and this contradiction implies the claim of the lemma.
	\end{proof}
	
	An argument  similar to the argument above about the order breakpoint of the domino graph will be used later to estimate the order breakpoint of infinitely presented groups.
	
	\begin{definition}\label{def:cactus}
		A {\it cactus graph} is a graph homeomorphic to an acyclic gluing of circles and intervals.
	\end{definition}
	
	\begin{lemma}\label{le:cactus}
		Let $\Gamma$ be a finite connected graph that does not contain a subgraph homeomorphic to the domino graph, then $\Gamma$ is a cactus graph. 
	\end{lemma}
	
	\begin{proof}
		Any tree is a cactus graph. If $\Gamma$ is not a tree, then $\Gamma$ contains a simple cycle $c$.  Observe that either
		any two vertices $A,B$ of this cycle belong to distinct connected components
		of $(\Gamma \setminus  c) \bigcup \{A,  B\}$, in this case we can argue by induction claiming that the connected components of $\Gamma \setminus c$ are acyclic gluings.
		Or there exist  $A,B$ on $c$, and a continuous path from $A$ to $B$ not passing through $c$, and this provides a homeomorphic image of a domino graph.
	\end{proof}
	
	Now we formulate a corollary of Lemmas \ref{lem:stargirth}, \ref{lem:domino} and \ref{le:cactus}.
	
	\begin{thm}\label{cor:finitegraphs}
		Let $\Gamma$ be a finite connected graph, with edges included. Then there are $3$ possibilities
		\begin{enumerate}
			\item $\Gamma$ is homeomorphic to an interval, then  $\mins(\Gamma) =2$.
			
			\item $\Gamma$ is a cactus graph, but not homeomorphic to an interval, then $\mins(\Gamma) =3$.
			
			\item Otherwise, $\mins(\Gamma) =4$.
		\end{enumerate}
	\end{thm}
	
	\begin{proof}
		If $\Gamma$ is homeomorphic to an interval, then the claim is straightforward.
		
		Otherwise, if $\Gamma$ does not contain a homeomorphic copy of a domino graph, then $\Gamma$ is a cactus graph, as follows from Lemma \ref{le:cactus}. 
		Observe that in this case either there is at least one circle in this wedge sum decomposition, or $\Gamma$ is a tree which contains an isometric  copy of a tripod.
		In both  cases, we know that for a subset of $\Gamma$  it holds that $\mins \ge 3$, and hence $\mins(\Gamma) \ge 3$. Also observe that $\mins(\Gamma) \le 3$, as follows from Corollary \ref{cor:acyclicgirth}. 	
		
		If $\Gamma$ contains a homeomorphic image of the domino graph, Lemma \ref{lem:domino} implies that $\mins(\Gamma) \ge 4$.  By Lemma \ref{lem:stargirth} we know that $\mins(\Gamma) \le 4$, and we can conclude that
		$\mins(\Gamma) = 4$.
		
	\end{proof}
	
	\section{Spaces and Groups with small order breakpoint}
	
	We recall that given a metric
	space $M$
	and an order $T$ on $M$, $\mins({M,T})$ is the minimal $s$ such that $\OR_{M,T}(s)<s$. One can define
	$\mins(M)$ as the minimum of $\mins({M,T})$, where the minimum is taken over all orders $T$ on $M$.
	
	In this section we will be mostly interested in uniformly discrete metric spaces. 
	Given a graph $\Gamma = (V,E)$, we consider the metric space $V$ with graph metric (all the edges have length $1$). 
	For a group $G$ with generating set $S$ we consider the word metric on $G$. 
	
	It is straightforward that for any metric space $M$ we have $\mins(M) \ge 2$.
	We know (see Lemma \ref{le:snakesqi}) that quasi-isometric uniformly discrete metric spaces have the same value of $\mins$.
	In particular, 
	for any finite metric space $M$, $\mins(M, T) = 2$ for any order $T$, because there are finitely many triples of points and elongation of snakes on $3$ points is bounded (see Lemma \ref{lem:snake}).  
	
	The following lemma characterizes spaces (quasi-isometric to geodesic ones) with this minimal possible value of $\mins$.
	
	\begin{lemma} \label{le:notrayline}
		Let $M$ be quasi-isometric to  a geodesic metric space, and not quasi-isometric to a point, a ray or a line. Then $\mins \ge 3$. If we assume moreover that $M$ is
		uniformly discrete, then $\mins = 2$ 
		if and only if $M$ is either bounded or quasi-isometric to a ray or to a line.
	\end{lemma}
	A particular case 
	for the last claim of  Lemma
	\ref{le:notrayline}
	above is when  $\Gamma = (V,E)$ is a connected graph, finite or infinite (as we have mentioned, a convention of this section is that the length of edges is equal to $1$).
	Then the order breakpoint $\mins(\Gamma) \leq 2$ if and only if the graph is quasi-isometric to a point, a ray or a line.  
	Moreover, it can be shown in this case that if  $\Gamma$ is not quasi-isometric to a point, a ray or a line, then for any order $T$ the space $(V, T)$ admits an $(\infty, \rm{Bounded})$-sequence of snakes on $3$ points.
	
	Before we prove this lemma, we formulate a discrete version
	of Lemma \ref{le:examplecircle} and Corollary \ref{cor:tripod}, which we will use also later in this section.
	\begin{lemma}\label{lem:discretecircletripod}
		\begin{enumerate}
			\item Let $M$ be a set of $n$ points of a
			circle of length $L$, with distance $L/n$ between consecutive points.
			Let $T$ be an order on $M$. Then
			$(M,T)$ admits a snake $(x_1, x_2, x_3)$ on three points,
			of diameter at least $L/2 -L/n$ and of
			width at most $L/n$.
			
			\item Consider a tripod with segments of length $L$. Let $N$ be a set of $3n+1$ points, containing the center of the tripod and $n$ points on each segment of the tripod, with distance between consecutive points $L/n$.
			Let $T$ be an order on $N$. Then $(N,T)$ admits a snake on $3$ points, of diameter at least $L-L/n$ and of width at most $L/n$.
		\end{enumerate}
	\end{lemma}
	\begin{proof}
		The claims of the lemma can be proven in the same way as Lemma \ref{le:examplecircle} and Corollary \ref{cor:tripod}. To avoid referring to these proofs, we observe that it follows from the claims of this lemma
		and this corollary. Indeed,
		if $T_1$ and $T_2$ are some orders on $M$ and $N$, consider the corresponding $\Star$ orders on the circle (which is the union of points of $M$ and edges between neighbouring points) and on the tripod.
		Let us call these star orders $T'_1$ and $T'_2$. By Lemma \ref{le:examplecircle} we know that $(M,T'_1)$ admits a snake on three points in $\varepsilon$-neighborhoods of two antipodal points, and by Corollary \ref{cor:tripod} we know that $T'_2$ admits a snake on $3$
		points of diameter close to $L$ and of width at most $L/n$. The claims for $T_1$ and $T_2$ follow by mapping each point in each snake to the center of its star figure (noting that the points are in distinct star figures).
	\end{proof}

	The statement of Lemma \ref{lem:containstripod} below can be obtained
	as a corollary of a result of
	M.~Kapovich  \cite{kapovich}, see
	Appendix \ref{sec:tripods}.
	
	In this lemma we use the following notation for tripods. A tripod $T_R$ consists of 3 segments of length $R$ glued
	together by an endpoint.
	
	\begin{lemma}\label{lem:containstripod}
		There exists $C: 0<C < 1$
		such that the following holds. If $M$ is a geodesic metric
		space which is not quasi-isometric to a point, a ray or a line, then
		for any $\varepsilon>0$,
		$M$ admits
		a  quasi-isometric embedding of a
		tripod of arbitrarily large size. More precisely,
		for each $n\ge 1$ there exists $\rho_n: T_n \to M$ such that for any $x,y \in T_n$
		$$
		C d_{T_n}(x,y) - \varepsilon \le d_M(\rho_n(x), \rho_n(y)) \le 
		d_{T_n}(x,y) +\varepsilon
		$$
		One can choose $C=1/3$.
	\end{lemma}

	Now we prove Lemma  \ref{le:notrayline}.
	
	\begin{proof}
		The order breakpoint for $\mathbb{Z}$ or $\mathbb{N}$ is 2 (we can take the natural order).
		If a uniformly discrete metric space $M$ is quasi-isometric to  $\mathbb{Z}$ or $\mathbb{N}$, then $\mins(M) = 2$ because of Lemma \ref{le:snakesqi}.
		Any uniformly discrete space that is quasi-isometric to a point has $\mins = 2$.
		
		Now let $M$ be a geodesic metric space that is not quasi-isometric to a point, ray or line, and let $T$ be an order on $M$.
		 	Now observe that from Lemmas \ref{lem:uniform}, \ref{lem:containstripod} and \ref{lem:discretecircletripod} it follows that $(M,T)$ admits an $(\infty, \rm{Bounded})$-sequence of snakes on 3 points.
		Using Lemma \ref{lem:uniform} we can extend this for metric spaces that are not necessarily geodesic but are quasi-isometric to a geodesic space.
		Then $\mins(M)\ge 3$ follows from Lemma \ref{lem:snake}.		
	\end{proof}
	
	A graph with $\mins \leq 3$ does not need to be quasi-isometric to a tree.
	Indeed, consider a ray and for each $i$ glue a circle of length $i$ at the point at distance $i$ from  the base point.
	This is a particular case of an acyclic
	gluing constructions
	(in this case applied to parts that are circles and intervals),
	considered in Section $2$ (see  Lemma  \ref{le:ORacyclicunion}). 
	We can choose an order on each part such that  the elongations of snakes on $4$ points are uniformly bounded from above in each part.
	From Corollary \ref{cor:acyclicgirth} it follows that for this acyclic gluing we have $\mins = 3$.
	
	However, groups with the word metric on their vertices satisfying  $\mins \le 3$ are virtually free and, hence, quasi-isometric to trees. We start with the following observation.
	
	\begin{lemma} \label{le: treesmins} 
		Let $M$ be a uniformly discrete metric space which is  quasi-isometric to a tree. Then there exists an order $T$ such that
		$\OR_{M,T}(3) <3$.
	\end{lemma}
	
	\begin{proof}
		It follows immediately from Lemma \ref{lem:hierarchicaltrees} and the second claim of Lemma \ref{le:snakesqi}.
	\end{proof}
	
	Let $G$ be a group with a finite generating set
	$S \subset G$, and let $\Gamma(G,S)$ be the Cayley graph of $G$ with respect to $S$. 
	The {\it number of ends} of a group is the supremum, taken over all finite
	sets $V$, of the  number of infinite connected components of $\Gamma \setminus V$.
	It is not difficult to see that the number of ends does not depend on the choice of the generating set. Moreover, it is a quasi-isometric invariant
	of groups. The number of ends of an infinite  group can be equal to $1$,
	$2$ or $\infty$. The notion of ends and their numbers can be more generally defined for metric spaces, or even more generally for topological spaces,
	but this notion has particular importance in group theory due to Stallings theorem that we will recall in the proof of
	Corollary \ref{cor:oneendedsubgroup}.
	For generalities about the number of ends see for example
	\cite{drutukapovichbook}.

	\begin{lemma}[One-ended groups have snakes on $4$ points] \label{le:oneended}
		Let $G$ be a one-ended finitely generated group.
		Then for any order $T$ and any integer $N$ in $(G, T)$ there exists a snake on $4$ points  of width $1$ and of diameter at least $N$. 
		In particular, $\mins(G) \ge 4$.
	\end{lemma}
	
	\begin{proof}
		Assume the contrary: for some order $T$ and integer $N_0$ there are no snakes on $4$ points of width $1$ and diameter at least $N_0$.
		Take a point $x \in G$ and consider a ball $B(x,2N_0)$ of radius $2N_0$, centered at $x$. 
		Since $G$ is one-ended,
		the complement of this ball has (in the Cayley graph $\Gamma(G,S)$ for some generating set $S$) one infinite connected component, which we denote by $Y$, and finitely many finite connected components.
		Take $N_1$ such that all these finite connected components belong to the ball $B(x,N_1)$, and let $N_2$ be the number of points in a ball of radius $N_1$.
		Note that for any two points $a, b\in G$ if 
		$d_G(a,b) > N_1$ then $b$ lies in the unique infinite connected component of $\Gamma(G,S) \setminus B(a, 2 N_0)$ (this unique connected component is a translation of $Y$).
		
		We recall that the {\it growth function} of a group $G$ with respect to a generating set $S$ is the cardinaly $\#B(e,n)$ of the ball of radius $n$ in the word metric of $(G,S)$.		
		Observe that all groups of linear growth are virtually $\Z$, and thus have
		two ends. We know therefore that  the  growth function of $G$
		satisfies  $\#B(x,n)/n  \to \infty$,
		as follows from an elementary case of the Polynomial Growth theorem, due to Justin (see  \cite{justin71}, see also \cite{MannHowgroupsgrow}).
		
		Now we fix a sufficiently large $R$. It is sufficient to assume that
		$\#B(x,R)/R > 2 N_2 +2$, $R>N_2$.
		Consider any two points $x, y$ of $G$ at distance  larger than $3R$.
		We define $M=\#B(x,R)$, observe that $M \geq R (2N_2+2)$. 
		Enumerate the points of $B(x,R)$ with respect to the order $T$:  $A_1,A_2,\dots, A_M$;
		$A_1 <_T A_2 <_T A_3 \dots <_T A_M$.
		It is clear that  $A_1$ and $A_M$ can be connected with a path of length not
		greater than $2R$.
		There are two neighbouring points  along this path such that their indices differ by at least  
		$
		\frac{M-1}{2R}> N_2$.
		We denote these points by $z$ and $w$, assuming
		that $z<_T w$.
		Observe that $G$ (in fact the ball
		$B(x,R)$) contains  at least $N_2$ elements $t$ such that $z <_T t <_T w$. 
		The infinite connected component of $\Gamma(G,S) \setminus B(z,2N_0)$ contains at least one such $t$.
		
		Observe that in this case  all elements in this connected  component are between $z$ and $w$ with respect to  $T$. 
		Indeed, otherwise in $\Gamma(G,S) \setminus B(z,2N_0)$ there exist two vertices at distance $1$ such that one of them is $T$-between $z$ and $w$ and the second is either $<_T z$ or  $>_T w$.
		This would imply that  in $G$ there is a snake on $4$ points with width $1$ and diameter $\geq N_0$, this is in  a contradiction with our assumption.
		In particular, for any point $t$ in $B(y,R)$ we have 
		$z <_T t <_T w$.
		
		But in the same way we can prove that in $B(y,R)$ there are two points $z',w'$ at distance $1$ such that for any $t \in B(x,R)$ it holds that $z' <_T t <_T w'$. 
		So we see that all points of $B(y,R)$ are between some two points
		of $B(x,R)$, and all points of $B(x,R)$ are between some two points
		of $B(y,R)$. Since these two balls do no intersect, we obtain a contradiction.
	\end{proof}

	\begin{corol}\label{cor:oneendedsubgroup}
		If a finitely generated group $G$ contains a one-ended (finitely generated) subgroup, then $\mins(G) \ge 4$. In particular, this inequality 
		holds for any finitely presented not virtually free group.
	\end{corol}
	\begin{proof}
		Let $H$ be a finitely generated one-ended subgroup of $G$ and $T$ an order on $G$.
		By Lemma \ref{le:oneended} we know that
		$H$ admits an $(\infty, \rm{Bounded})$-sequence of snakes on $4$ points.
		Observe that the embedding map of any finitely generated subgroup into an ambient group is uniform.
		We can therefore apply Lemma \ref{lem:uniform} to conclude that for any order $T$ the space $G$ admits a sequence of snakes of 
		bounded width on $4$ points, and therefore that  $\mins(G) \ge 4$.
		
		Now we explain the second claim of the corollary about finitely presented
		groups.
		Stallings theorem (see e.g.
		\cite{drutukapovichbook})
		shows that if a finitely generated group has at least  two ends, then it is an amalgamated free product over a finite group or an HNN extension over a finite subgroup.
		For any group which we obtain in this way if it has at least two ends, one can again write it as an amalgamated free product or an HNN extension over a finite subgroup.
		We recall that a group is said to be {\it accessible} if this process terminates. That is, if the group is accessible, then it is the fundamental group of a finite graph of
		group such that each of the vertex groups is either finite or has one end, and each of the edge groups is finite, see \cite{dunwoody}, \cite{serreTrees}.
		If all vertex groups
		are finite, the group is virtually free (see Proposition 11,
		Section 2.6 \cite{serreTrees}). 
		Since the vertex groups are subgroups of the fundamental group of a graph of groups, it is clear that if a group is accessible and not virtually free, then
		it contains a one-ended subgroup. By a result of Dunwoody any finitely presented group is accessible
		(see \cite{dunwoody}, see also \cite{drutukapovichbook}). Therefore, any finitely presented not virtually free group admits a one-ended subgroup.
	\end{proof}

	Any  infinitely presented group
	admits isometric embeddings of arbitrarily large
	cycles. 
	See Theorem A in \cite{Nima}, which states that shortcut groups
	(they are by definitions those that act properly and cocompactly on graphs that do not admit isometric
	embeddings of large cycles)  are finitely presented.
	In particular, this result shows that Cayley graphs of not finitely presented groups have long cycles. 
	This is formultated in the first part of the Lemma below, for the  convenience of the reader we include the proof.
	
	\begin{lemma}[Cycles and Trimino graphs in infinitely presented groups]\label{le:infprcycle}
		Let $G$ be a group with a finite generating set $S$. Assume that  $G$ is not finitely presented.
		
		\begin{enumerate}
			\item Then for any $N>0$ there is an isometric embedding of a cycle with length more than $N$ into the Cayley graph $\Gamma(G,H)$.
			\item For any $N_0$ and $N_1>0$ there exist  three isometric embeddings of cycles $\omega_1$, $\omega_2$, $\omega_3$ in such a way that $\omega_1$ and $\omega_2$ have a common path
			$AB$, $\omega_2$ and $\omega_3$ have a common path
			$CD$, $AB$ and $CD$ have lengths $\ge N_0$; the ratio of the length of $BC$ or $AD$ to that of $\omega_1$ or $\omega_3$ is $\ge N_1$.
			(see Figure \ref{pic:figureeight}).
		\end{enumerate}
		
	\end{lemma}
	
	\begin{figure}[!htb] 
		\centering
		\includegraphics[scale=.75]{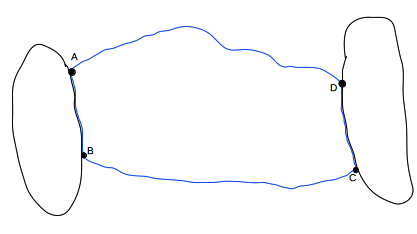}
		\caption{Both black closed paths ($\omega_1$ and $\omega_3$) and the blue closed path (that is $\omega_2$) are isometrically embedded cycles. The length of both grey-and-blue intervals are much smaller than the lengths of both black intervals, which are in its turn are much smaller than the length of both blue intervals.}
		\label{pic:figureeight}
	\end{figure}
	
	\begin{proof}
		$(1)$ Take in $G$ a  cyclically irreducible relation $u = s_1s_2\dots s_n$, $s_i\in S$ such that $u$ is not a consequence of shorter relations of $G$ (since $G$ is not finitely presented, there are infinitely many such $u$).
		Consider in $\Gamma(G,S)$ a path starting at an arbitrary vertex and having labelling $u$.
		This is a cycle $c_u$ with length $n$, assume it is not an isometric embedding.
		Then some two vertices $v_i$ and $v_j$ are connected with a path $p$ such that $p$ has fewer edges than both parts of $c_u$ between $v_i$ and $v_j$. 
		If we keep one of these parts and replace another one by $p$, we obtain two relations in $G$ that are shorter than $u$ and imply $u$. This contradiction
		implies the first claim.
		
		$(2)$ To prove the second part of the lemma, it is sufficient to take an isometric embedding of a long cycle, choose another, much longer isometrically embedded cycle that has a long common word with  the first one, and finally take a third, long but much shorter than the second one isometrically embedded cycle which has a long almost antipodal common subword with the second one.
		
		More precisely, we argue as follows. For a word $u$ in the alphabet
		$S$ we denote by $l_S(u)$ its length.
		From $(1)$ we know that there exists an infinite sequence of words $(u_i)$ in
		the alphabet $S$ such that $\lim_{i\to \infty} l_{S}(u_i) = \infty$ and for any $i$ and any $g\in G$, the path in $\Gamma(G,S)$, starting at $g$ and
		labelled $u_i$
		is an isometric image of a cycle of length $l_{S} (u_i)$.
		
		Any word $u_i$  of length greater than $2N_1N_0$ can be decomposed as follows:
		\[
		u_i = x_i a_i y_i b_i,
		\]
		where $l_S(x_i) = l_S(y_i) = N_0$ and $|l_{S}(a_i) - l_{S}(b_i)| \leq 1$. 
		In other words, we chose this decomposition in such a way that two segments labeled by $x_i$ and $y_i$ are almost opposite in the cycle labeled by $u_i$.
		
		\noindent Since the number of such pairs $(x_i, y_i)$ is bounded by $(\#S)^{2N_0}$, there exist $i$ and $j$ such that $x_i = x_j$, $y_i = y_j$,  and $l_S(u_j) > 4N_1l_S(u_i)$.
		
		\noindent In $\Gamma(G,S)$ there are vertices $A, B, C, D$ and six paths (lower indices denote starting and ending points) $p_{AB}$, $p_{BA}$, $p_{BC}$, $p_{CD}$, $p_{DC}$, $p_{DA}$ with label words 
		$x_i$, $a_iy_ib_i$, $a_j$, $y_i$, $b_ix_ia_i$, $b_j$ correspondingly (see Figure \ref{pic:figureeight}).
		
		The cycles $p_{AB}p_{BA}$, $p_{AB}p_{BC}p_{CD}p_{DA}$ and $p_{CD}p_{DC}$ have labels $u_i$, $u_j$ and a cyclic shift of  $u_i$ correspondingly, they are isometrically embedded,
		$p_{BC}$ and $p_{DA}$ are geodesics, 
		and all the required inequalities hold.
	\end{proof}
	
	\begin{thm}\label{thm:girth4}
		Let $G$ be a finitely generated group.
		Then $\mins(G) \leq 3$ if and only if $G$ is virtually free.
	\end{thm}
	
	\begin{proof}
		We know that  free groups have  $\mins(G,T) \leq 3$,
		see Remark \ref{rem:freegroup}. 
		In fact, $\mins = 3$ if
		the number of generators is at least $2$ (see Lemma \ref{le:notrayline}), and is $2$ otherwise (if $G$ is $\Z$).
		Since $\mins$ is a quasi-isometric invariant of groups (see Lemma \ref{le:snakesqi}), we can conclude that any virtually free group $G$ satisfies $\mins(G) \leq 3$.
		
		Now we have to prove the non-trivial claim of the theorem. That is, to prove that if $G$ is not virtually free and $S$  is a finite generating set, then  for any order $T$ it holds that $\mins({G,T}) \ge 4$.
		
		\noindent The second claim of Corollary \ref{cor:oneendedsubgroup} says that if $G$ is finitely presented but not virtually free, then $\mins({G,T}) \ge 4$ for any order $T$. 
		
		Therefore, it is enough to assume now  that $G$ is not finitely presented. In this case we consider a mapping of the ``trimino figure'', guaranteed by Lemma \ref{le:infprcycle}, and search for an appropriate snake in the image of this figure. (If we would have a quasi-isometric embedding of a domino figure, 
		with  length of its edges large enough, in our group
		then we could use directly a discrete version of the domino Lemma \ref{lem:domino}. Since we do not have this information, it is more convenient for us to work with the trimino figure).
		
		Suppose that the claim we want to prove is not satisfied: $G$ is a group which is not finitely presented, $S$ is a finite generating set, $T$ is an order on $G$ and
		there are no snakes on $4$ points
		with large elongation in $(G,S)$, with respect to the order $T$. In particular, this means
		that for  some $K$ there are no snakes on $4$ points  in $(G,T)$ with width $1$ and of diameter greater than or equal to $K$.
		
		Apply the  second part of Lemma \ref{le:infprcycle} and consider  
		the image of the  ``trimino figure'' satisfying the claim of this lemma 
		for $N_0 = 10K$, $N_1 = 10$.
		
		The cycles $p_{AB}p_{BA}$, $p_{AB}p_{BC}p_{CD}p_{DA}$ and $p_{CD}p_{DC}$,
		constructed in this lemma, are isometrically embedded, 
		$l_{G,S}(p_{AB}) \ge 10K$, $l_{G,S}(p_{CD}) \ge 10K$,  $\min(l_{G,S}(p_{BC}),l_{G,S}(p_{DA})) > 10 \max(l_{G,S}(p_{BA})+ l_{G,S}(p_{AB}), l_{G,S}(p_{CD})+ l_{G,S}(p_{DC}))$.
		
		 Since $p_{AB}$ and $p_{CD}$ are geodesics, 
		the cycles $p_{AB}p_{BA}$ and $p_{CD}p_{DC}$ have length $\ge 20K$.		
		From the first claim of Lemma \ref{lem:discretecircletripod}
		we know that 
		the cycle $p_{AB}p_{BA}$ contains points $v_1, v_2, v_3$ 
		such that $v_1 <_T v_2 <_T v_3$, $d(v_1, v_3) = 1$, $d(v_1, v_2) > 9 K$.
		
		We claim that in $\Gamma(G, S)$ there is a (not necessarily directed) continuous path $p$  from $v_2$ to the vertex $C$ such that it does not intersect with the ball $B(v_1, 2K)$.
		
		\noindent First observe that the distance between the cycles $p_{AB}p_{BA}$ and $p_{CD}p_{DC}$ is 
		$$\ge \min(l_{G,S}(p_{BC}),l_{G,S}(p_{DA})) - \left(l_{G,S}(p_{BA})+ l_{G,S}(p_{AB})+ l_{G,S}(p_{CD})+ l_{G,S}(p_{DC})\right) > 50K.
		$$
		
		Hence $B(v_1, 2K)$ does not intersect with $p_{CD}p_{DC}$, so reaching any point of $p_{CD}p_{DC}$ is enough for our purpose.
		Since the cycle $p_{AB}p_{BA}$ is isometrically embedded, 
		its intersection with $B(v_1, 2K)$ is a geodesic segment of length $4K$, denote this segment by $p_s$.
		It is  clear that $v_2\not\in p_s$.
		
		\noindent We also observe that $B(v_1,2K)$ can not intersect with both $p_{BC}$ and $p_{DA}$. 
		Indeed, since the cycle $p_{AB}p_{BC}p_{CD}p_{DA}$ is
		isometrically embedded, the distance in $\Gamma(G,S)$ 
		between the subsets $p_{BC}$ and $p_{DA}$ is equal to 
		$\min(l_{G,S}(p_{AB}), l_{G,S}(p_{CD})) \geq 10K$.
		Without loss of the generality we assume that the set $B(v_1, 2K)$ does not intersect with $p_{BC}$.
		In particular, $B \not \in p_s$. 
		Then there is a path $p_{v_2B}$ from $v_2$ to $B$ such that 
		vertices of it belong to $p_{AB}p_{BA} - p_s$.
		Denote by $p$ the path $p_{v_2B}p_{BC}$, this path satisfies the property we were looking for.
		
		Consider some points
		$x_1,x_2 \in G \setminus B(v_1,2K)$. If $d(x_1, x_2) = 1$ and $v_1 <_T x_1 <_T v_3$ 
		then we claim that $v_1 <_T x_2 <_T v_3$, otherwise we find a snake on $4$ points with width $1$ and diameter greater than $K$.
		So, all points of the path $p$ are between $v_1$ and $v_3$, and all points of the circle $p_{CD}p_{DC}$ are greater than $v_1$ (in this sentence ``between'' and ``greater'' refers to the order $T$).
		
		But in the same way we can prove that for some point $v_4\in p_{CD}p_{DC}$ and any point $v$ of $p_{AB}p_{BA}$ we have 
		$v_4 <_T v$, this leads to a contradiction.
	\end{proof}

	\section{Hyperbolic spaces} \label{sec:hyper}
	
	It is known that $\delta$-hyperbolic spaces behave nicely with respect to some questions related to the travelling
	salesman problem. See Krauthgamer and Lee \cite{KrauthgamerLee}  who show that under a natural local assumption on such spaces there exists an
	approximate randomized algorithm for the travelling salesman problem. 
	
	The goal of this section is to prove Theorem \ref{thm:hyperbolicspace}, showing that there is a very efficient order on these spaces to solve the universal travelling salesman  problem.
	In contrast with the context of \cite{KrauthgamerLee},  the constant in the bound for the order ratio function can not be close to one even in the case of metric trees (unless the tree is homeomorphic to a ray, line or an interval, otherwise it clearly admits an embedded tripod).

	In the introduction it was already mentioned that in view of the theorem of  Bonk and Schramm, our main step in the proof of Theorem \ref{thm:hyperbolicspace} is to prove it for uniformly discrete spaces that are quasi-isometric to hyperbolic spaces $\H^d$. 
	We work with a graph that corresponds to a specific tiling of $\H^d$ (also shown on Figure 
	\ref{pic:tilingwithtree2}  for this tiling in the case of $\H^2$).
	The order we define on this graph comes from a  hierarchical order on a certain associated tree (see Figure \ref{pic:tilingwithtree2}). It is well known that the metric of any finite subset of a $\delta$-hyperbolic space can be well approximated by the metric of a tree. A comparison lemma (see for example \cite{ghysdelaharpe}) states that given a subset of cardinality $k$ in a $\delta$-hyperbolic space, its metric can be approximated by the metric of a finite tree
	up to an additive logarithmic error $ \delta O(\ln k)$.
	As we have already mentioned in the introduction, this metric approximation is not enough for our purposes.
	We will work with some trees,  which in some other sense approximate  our space,
	associate an order to these trees and we need check some properties of these trees and their associated orders.

	\subsection{Definitions and basic properties of $\delta$-hyperbolic spaces}
	
	A metric space is said to be  {\it $\delta$-hyperbolic} if there exists $\delta\ge 0$ such that the following holds. For any points $X_1$, $X_2$, $X_3$, $X_4$ consider three numbers
	$|X_1X_2| +|X_3X_4|$, $|X_1X_3|+|X_2X_4|$, $|X_1X_4|+|X_2X_3|$.  Here we use the notation $|X_iX_j|=
	d_M(X_i,X_j)$.
	We require that the difference between the largest of these three numbers and the middle one is at most $\delta$.
	It is well-known that for geodesic metric spaces  this definition (with appropriate choice of $\delta$) is equivalent to several other possible definitions, one of them is in terms of thin triangles. A geodesic triangle is said to be {\it $\delta$-slim} if each of its sides is contained in the $\delta$-neighborhood of the union of the two other sides.
	A geodesic metric space is $\delta$-hyperbolic if there exists $\delta$ such that any geodesic triangle is $\delta$-slim.
	Any Hadamard manifold (complete simply connected Riemannian manifold of sectional curvature $\le \kappa <0$) is $\delta$-hyperbolic for some $\delta$ depending on $\kappa$ (see e.g. Chapter $3$ in \cite{ghysdelaharpe}).
	In particular, a hyperbolic space $\mathbb{H}^d$ is $\delta$-hyperbolic for some $\delta>0$.
	We recall that a {\it quasi-geodesic} is the image of an interval with respect to a quasi-isometric embedding.
	If this quasi-isometry has multiplicative constant $A$ and additive constant $B$, we say that such quasi-geodesic is an $(A,B)$-quasi-geodesic.

	If a metric space is geodesic, a basic property of a $\delta$-hyperbolic space is
	that all $(A,B)$-quasi-geodesics between points $X$ and  $Y$  lie at a bounded distance from any geodesic between $X, Y$, where this distance is estimated in terms of $\delta$, $A$ and $B$.
	This statement is sometimes called the Morse Lemma.
	
	We also recall that if $M_1$ and $M_2$ are quasi-isometric geodesic  metric spaces and $M_1$ is $\delta$-hyperbolic, then there exists $\delta'$ such that $M_2$
	is $\delta'$-hyperbolic.
	For these and other basic properties of $\delta$-hyperbolic spaces see e.g. \cites{ghysdelaharpe,  BridsonHaefliger}. For optimal estimates of the constants in the Morse Lemma see
	\cites{schur, gouezelschur}.

	\subsection{Binary tiling}\label{subsec:bin}
	
	There is a well-known tiling of the hyperbolic plane $\H^2$, which is called
	{\it the binary tiling} and also sometimes called the B{\o}r{\o}czky tiling. A version of this tiling that we describe below
	exists in $(d+1)$-dimensional
	hyperbolic spaces $\H^{d+1}$, $d \geq 0.$
	
	Consider the space  $\R^{d+1}$, the coordinates of this space we denote by $(x_0, x_1, \dots, x_d)$.
	Denote by  $H^{d+1}$ its half-space with  $x_0 > 0$.
	We will use the words {\it up} and {\it down} to refer to increases and decreases in the coordinate
	$x_0$. We also call  changes of this coordinate {\it vertical}, and changes of all the other coordinates {\it horizontal}.
	We subdivide $H^{d+1}$ (viewed as a subset of Euclidean space $\R^{d+1})$ into cubes: for integers $k, a_1, a_2, \dots, a_d$ consider the set of points satisfying
	\[
	\begin{cases} 2^{k} \leqslant x_0 \leqslant 2^{k+1} 
	\\ 2^k a_1 \leqslant x_1 \leqslant 2^k (a_1 + 1) \\
	\dots \\
	2^k a_d \leqslant x_d \leqslant 2^k (a_d + 1) \end{cases}
	\]
	
	We call such a cube  {\it a tile with coordinates} $(k,a_1,\dots, a_d)$,  and when we speak about {\it faces} of this cube we refer to its $d$-dimensional faces. 
	(This tiling of  $\mathbb{R}\times\mathbb{R}_+$, that is for  $\H^2$ and   $d = 1$, is shown on Figure \ref{pic:tilingwithtree2}).
	
	We call the subset  $2^{k} \leqslant x_0 \leqslant 2^{k+1}$  
	{\it the layer of the level $k$}. The layer of level $k$ consists
	of $(d+1)$-dimensional cubes  of the same size, with the sides of length $2^k$. The centers of these cubes form a standard
	Euclidean $d$-dimensional lattice. 
	A tile  with coordinates $(k,a_1,\dots,a_d)$ is adjacent to  $2d$ tiles of the same level  with coordinates $(k, a_1,\dots, a_{i-1}, a_i\pm 1, a_{i+1}, \dots, a_d)$, it is also adjacent to  one of the upper level, namely to $(k+1, \lfloor a_1/2\rfloor,\dots, \lfloor a_{d}/2\rfloor)$ and it is adjacent to  $2^d$ tiles of the lower level of the form $(k-1, 2a_1 + e_1, \dots, 2a_d + e_d)$, where each  $e_i$ is equal to $0$ or $1$.

	Consider the graph $\Gamma_{d+1}$, the vertices of which correspond to (centers of) tiles, and
	two vertices are joined by an edge if the tiles are adjacent. 
	We know that the degree of each vertex
	of  $\Gamma_{d+1}$ is equal to  $2d + 2^d + 1$.

	We define a metric on the space $H^{d+1}$ by putting 
	$$(ds)^2 = \frac{(dx_0)^2 + (dx_1)^2 + \dots + (dx_d)^2}{x_0^2}.
	$$
	We identify $H^{d+1}$ together with this metric with the hyperbolic space $\H^{d+1}$ (half-space model). 
	The distance between two points
	$x = (x_0, x_1, \dots, x_d)$ and 
	$x' = (x'_0, x'_1, \dots, x'_d)$ 
	is then 
	$$
	d_{\H^{d+1}}(x,x') =2\ln
	\frac{\sqrt{\sum_{i=0}^d(x_i-x'_i)^2} +\sqrt{(x_0+x'_0)^2 + \sum_{i=1}^d(x_i-x'_i)^2}}
	{2\sqrt{x_0 x'_0}}.
	$$
	
	Observe that the   graph $\Gamma_{d+1}$ is quasi-isometric to  $\H^{d+1}$. To see this observe that
	the mapping from  $\H^{d+1}$ to $\H^{d+1}$ 
	\[(x_0,\dots,x_d) \to (2^k x_0, 2^k x_1 + 2^ka_1, \dots, 2^k x_d + 2^ka_d)
	\]
	is an isometry of $\H^{d+1}$, as follows from the above mentioned definition of the hyperbolic metric.
	Hence any tile  can be mapped to any other tile of our tiling by an
	isometry of $\H^{d+1}$. 
	Observe also that each tile has a bounded number of faces, and that each point is adjacent to a bounded number of tiles.
	Therefore, our graph is quasi-isometric to $\H^{d+1}$.
	
	\begin{definition}
		For given two vertices $\alpha$ and $\omega$ of the graph $\Gamma_{d+1}$ we consider a path of the following form.
		This path makes from $\alpha$ several (possibly none) steps   upwards to some vertex $z$.
		After this it makes several horizontal steps to some vertex $t$ that is close to $z$: close in the sense that the tiles corresponding to $z$ and $t$ have a common point. 
		We require that  this segment of the path has minimal length among all horizontal paths connecting $z$ and $t$ in $\Gamma_{d+1}$.  
		After that the geodesic makes
		several steps (possibly none) downwards from $t$ to $\omega$. 
		We will call paths of the form described above
		{\emph{up-and-down paths}}. 
		The first group of edges of the path is called its \emph{upwards component}, the last group of edges is called its \emph{downwards component}.
		We call an up-and-down path  which has a minimal number of vertical edges \emph{ a standard up-and-down path.}
		
	\end{definition}
	
	Such standard up-and-down paths are essentially unique, that is, their ``upwards'' and ``downwards'' components
	are uniquely defined, and the only freedom is the choice of at most $d$ horizontal  steps.
	
	\begin{lemma}[Standard quasi-geodesics]\label{le:up_and_down}
		There exist $A, B>0$, depending on the dimension $d$, such that for any two vertices $\alpha$ and  $\omega$ of the graph $\Gamma_{d+1}$ a standard up-and-down path between them is an $(A,B)$-quasi-geodesic in $\Gamma_{d+1}$.
	\end{lemma}
	\begin{proof}
		Observe that a subpath of an up-and-down path is also an up-and-down path, and that this subpath is clearly standard if the initial one is standard. Therefore in order to show that standard up-and-down paths are quasi-geodesics,
		it is sufficient to show that for all $\alpha$ and $\omega$ in $\Gamma_{d+1}$, such paths between $\alpha$ and $\omega$
		have length at most $C L$, where $L$ is the distance between $\alpha$ and $\omega$ and $C$ depends only on $d$.

		Observe that a standard quasi-geodesics between two points can be obtained in the following way:
		given two points   $\alpha$ and  $\omega$, if the level of $\omega$ is greater than or equal than that of $\alpha$,  we first move
		from $\alpha$ upwards  to some point $\alpha'$ of the same level as $\omega$. Then we start moving upwards (simultaneously) from $\alpha'$ and $\omega$ until we reach for the first time the points of the same level that are close in the sense described in the definition of
		up-and-down paths, see Figure \ref{pic:path}.
		
		Let $m$ be the shortest length of a path among paths that move only horizontally from $\alpha'$ to $\omega$.
		Observe that the length of a standard up-and-down path between $\omega$ and $\alpha'$  
		is at most $d + 2\log_2 m$.
		The length of a standard up-and-down path connecting $\alpha$ and $\alpha '$ is clearly the difference of their levels $k_{\alpha'} - k_{\alpha }$.  
		This shows that, for some positive $C_1$ depending on $d$, the length of a standard up-and-down path 
		between $\alpha$ and $\omega$ in $\Gamma_{d+1}$ is  at most $C_1 \ln (m+1) +k_{\alpha '} -k_\alpha=
		C_1 \ln (m+1) + k_\omega - k_\alpha$. 
		
		On the other hand, for any path between $\omega$ and $\alpha$ in $\Gamma_{d+1}$, observe that its length $L$ satisfies $L \ge  k_\omega - k_\alpha$.
		From the distance formula of the half-space model we have 
		
		$$
		d_{\H^{d+1}}(\alpha,\omega) \ge C_2 \ln(m+1),
		$$
		
		where $C_2>0$ depends only on $d$. 
		Since $\Gamma_{d+1}$ and $\mathbb{H}^{d+1}$ are quasi-isometric, we also conclude that  
		$L \ge C_3 \ln (m+1)$ for some $C_3>0$ depending only on $d$.
		Hence the length of our up-and-down path is $\le (\frac{C_1}{C_3}+1)L$.
	\end{proof}
	
	We will also call standard up-and-down paths 
	{\it  up-and-down quasi-geodesics} and {\it standard quasi-geodesics}.

	If we relax the closeness condition in the definition of up-and-down paths by saying that $z$ and $t$ are of the same vertical level and at bounded distance (the bound depends on $d$), one can moreover show that a geodesic between any two points in $\Gamma_{d+1}$
	can be chosen in this standard form. 
	We do not use this statement in our paper, but we mention that it can lead to an alternative proof of Lemma \ref{le:up_and_down}. 
	We are grateful for the referee for indicating us that such proof is somehow analogous to Lemma 3.10 in \cite{manning}.

	Let us say that an edge in $\Gamma_{d+1}$ is {\it vertical} if this edge has a non-zero change of $x_0$ coordinate.
	Now consider a graph which  has the same vertex set as $\Gamma_{d+1}$ and where the edges consist of just the vertical edges of
	$\Gamma_{d+1}$. We denote this graph by $\overline{\Gamma_{d+1}}$.
	
	\begin{rem} The graph $\overline{\Gamma_{d+1}}$ is a forest. Indeed, observe that any point has only one vertical ascendant, hence this graph does not have cycles.
		This forest consists of $2^{d}$ disjoint trees. For example, in the case $d=1$, these two trees correspond to the two quadrants $x_1<0$,
		and $x_1>0$, in the upper half plane $x_0>0$.
	\end{rem}

	Denote by $\mathcal{F}_{d+1}$ one of the connected components of $\overline{\Gamma_{d+1}}$. 
	We know that $\mathcal{F}_{d+1}$ is a tree.  We can introduce an orientation on the edges of $\mathcal{F}_{d+1}$, saying that the edges are oriented downwards. 
	\begin{rem}
		Any finite set of vertices of $\Gamma_{d+1}$ can be moved to the  component $\mathcal{F}_{d+1}$ by an isometry of $\H^{d+1}$. 
	\end{rem}
	
	By Lemma \ref{rem:goedel} if we want to  find an order on  $\Gamma_{d+1}$ with bounded order ratio function, 
	it is sufficient to find an order with this property on the connected component $\mathcal{F}_{d+1}$. 
	We will check this claim for  orders on $\mathcal{F}_{d+1}$ with convex
	branches.
	
	In this section when we discuss branches of $\mathcal{F}_{d+1}$ we mean branches in the sense of an oriented tree:
	a branch of $\mathcal{F}_{d+1}$ corresponds to the set of tiles belonging to 
	
	\begin{equation} \label{eq:branch}
\begin{cases}  x_0 \leqslant 2^{k+1} 
\\ 2^k a_1 \leqslant x_1 \leqslant 2^k (a_1 + 1) \\
\dots \\
2^k a_d \leqslant x_d \leqslant 2^k (a_d + 1) \end{cases}
	\end{equation}

	for some choice of integers $a_1, \dots, a_d$ and $k$.
	
	Observe that Lemma \ref{le:convex} implies that  $\mathcal{F}_{d+1}$ admits  orders $T$ such that 
	branches are convex
	with respect to $T$.
	We also described explicitly such an order for a
	finite  tree in the beginning of Subsection \ref{subsec:trees}, this construction can easily be modified for the infinite tree $\mathcal{F}_{d+1}$. 
	
	Now we prove
	
	\begin{figure}
		\centering
		\begin{tikzpicture}[scale=0.8, every node/.style={scale=0.8}]
		\newcommand{\bintile}[5]{
			\begin{scope}[shift={(#1,#2)}, scale = 0.5*#3]
			\draw[fill=#5] (0,0) -- (0,1) -- (1,1) -- (1,0) -- (1,0) -- (0,0) -- cycle ;
			\draw (0.5,0.5)
			node{\scalebox{#3}{#4}};
			\end{scope}
		}
		
		\foreach \x in {0,...,3}{
			\bintile{4*\x}{4}{8}{}{white}
		}
		
		\foreach \x in {0,...,7}{
			\bintile{2*\x}{2}{4}{}{white}
		}
		
		\foreach \x in {0,...,15}{
			\bintile{\x}{1}{2}{}{white}
		}
		
		\foreach \x in {1,...,30}{
			\bintile{0.5*\x}{0.5}{1}{}{white}
		}
		\foreach \x in {2,...,61}{
			\bintile{0.25*\x}{0.25}{0.5}{}{white}
		}
		\bintile{2}{2}{4}{$x$}{gray!40!white}
		\bintile{2}{1}{2}{}{gray!40!white}
		\bintile{3}{1}{2}{}{gray!40!white}
		\foreach \x in {0,...,3}{
			\bintile{2 + 0.5*\x}{0.5}{1}{}{gray!40!white}
		}
		\foreach \x in {0,...,7}{
			\bintile{2 + 0.25*\x}{0.25}{0.5}{}{gray!40!white}
		}
		\filldraw [fill=gray!40!white] (2,0.) rectangle (4,0.25);
		
		\draw[dotted] (0,0) -- (16,0);
		\bintile{8}{2}{4}{}{green!20!white}
		\bintile{12}{2}{4}{}{green!20!white}
		\bintile{8}{4}{8}{$z$}{green!20!white}
		\bintile{12}{4}{8}{$t$}{green!20!white}
		\bintile{9}{1}{2}{}{green!20!white}
		\bintile{12}{1}{2}{}{green!20!white}
		\bintile{9}{0.5}{1}{$\alpha'$}{green!20!white}
		\bintile{12.5}{0.5}{1}{$\omega$}{green!20!white}
		\bintile{9}{0.25}{0.5}{$\alpha$}{green!20!white}
		\draw (0,8) -- (0,9);
		\draw (8,8) -- (8,9);
		\draw (16,8) -- (16,9);
		
		\end{tikzpicture}
		\caption{Branch of a tile $x$ and standard up-and-down path between tiles $\alpha$ and $\omega$.} \label{fig:box}
		\label{pic:path}
	\end{figure}
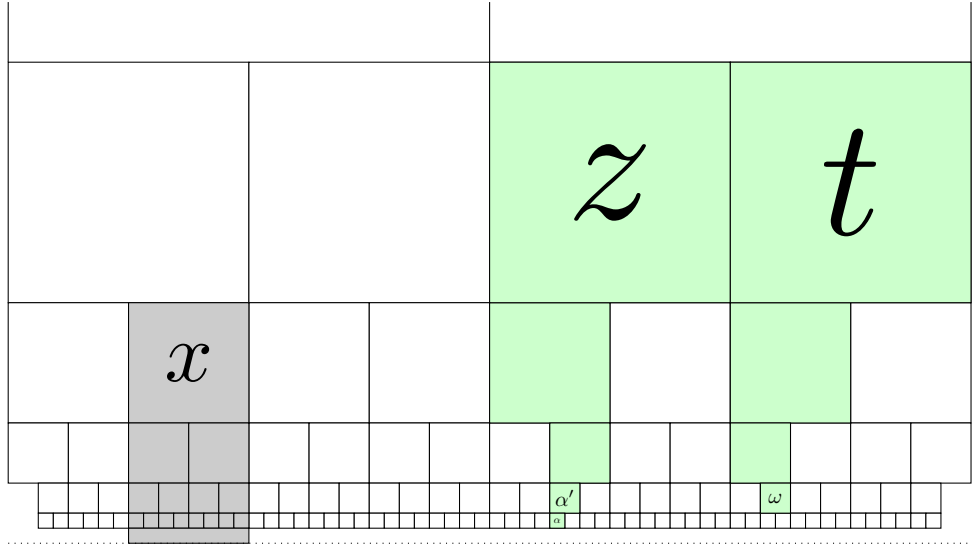
	
	\begin{lemma}\label{lem:insidethebox}[Standard paths are inside branches]
		Suppose that  $\alpha$ and $\omega$ belong to some branch of the tree $\mathcal{F}_{d+1}$. 
		Then the vertices of  any standard
		up-and-down path between $\alpha$ and $\omega$ also belong to this branch.
	\end{lemma}
	\begin{proof}
		Let $x$ be a vertex of $\mathcal{F}_{d+1}$ such that $\alpha$ and $\omega$ belong to the branch of $x$.
		Let 
		$$x = y_0, y_1, \dots, y_n = \alpha$$
		be the shortest path connecting $x$ and $\alpha$ and let 
		$$x = z_0, z_1, \dots, z_m = \omega$$
		be the shortest path connecting $x$ and $\omega$.
		If one of these paths contains the other, the statement of the Lemma is clear.
		
		Otherwise there exists a maximal index $k$ such that $y_k \neq z_k$ but the tiles corresponding to $y_k$ and $z_k$ have a common point.
		For any standard up-and-down path connecting $\alpha$ and $\omega$ its upwards component is $(y_n, y_{n-1}, \dots y_k)$ and its downwards component  is $(z_k, \dots, z_m)$.
		Consider arbitrary vertex $v$ of its horizontal component. Any of its coordinates is equal to the corresponding coordinate of $y_k$ or $z_k$, and from formula \ref{eq:branch} it follows that $v$ also belongs to the branch of $x$.
		The Lemma follows.
	\end{proof}

	\begin{lemma}[Multiplicity upper bound on paths with respect to hierarchical orders] \label{lem:multiplicity}
		Let  $T$ be an  order on $\Gamma_{d+1}$ such that all the branches of the tree $\mathcal{F}_{d+1}$ are convex with respect to $T$. 
		Consider distinct vertices
		$\alpha_0 >_T \alpha_1 >_T \dots >_T \alpha_n$
		and for each  $0\leqslant i < n$ assume that the sequence of tiles
		\[(u^i_0, u^i_1, u^i_2, \dots, u^i_{n_i}),
		\]
		with $u^i_0= \alpha_i$ and $u^i_{n_i} = \alpha_{i+1}$
		is a standard up-and-down quasi-geodesic, we denote it by $p_{i+1}$.
		Then each vertical edge   appears at most once among upwards edges of the above mentioned paths $p_1, \dots, p_n$
		and at most
		once among downwards edges.
		There exists $C$ depending only on $d$ such that any vertex belongs to at most $C$ paths among  $p_1, \dots, p_n$. 
	\end{lemma}
	
	\begin{proof}
		Let $e=(v,w)$ be some edge in $\mathcal{F}_{d+1}$, with $w$  one level lower than $v$.
		Suppose that this edge $e$  belongs to the upward part of  the standard up-and-down path $p_{i+1}$ between
		$\alpha_i$ and $\alpha_{i+1}$. Take $j>i$ and let us show that $e$ can not belong to $p_{j+1}$. 
		
		Indeed, $\alpha_i$ belongs to the branch of $w$. We know that the level of $v$ is higher than that of $w$, and hence
		$v$ does not belong to the branch of $w$. We can therefore
		conclude
		by Lemma \ref{lem:insidethebox} that $\alpha_{i+1}$ is also not inside the branch of $w$. 
		Since branches of the tree are convex with respect to our order, this implies that for  any $j\ge i+1$ the vertex $\alpha_j$ does not
		belong to the branch of $w$ (since by assumption of the lemma $\alpha_{i+1}$ is between $\alpha_i$ and $\alpha_j$, and since we know that $\alpha_i$ is in the branch, and $\alpha_{i+1}$ is not in the branch).
		This implies that
		any path going upward from $\alpha_{j}$ does not enter this branch. In particular, the edge $e$ is not on the standard up-and-down path from $\alpha_{i+1}$ to $\alpha_{i+2}$.
		The case of downwards paths is analogous.
		
		Now to prove the last claim of the lemma observe that it is sufficient to bound the multiplicity
		of all (not necessarily  vertical) edges on our path. (Indeed, any but the very last vertex of our path belongs to an edge of our path
		starting from this vertex).
		
		Consider a horizontal edge $f$. 
		Observe that each standard up-and-down quasi-geodesic that contains  $f$ is either {\it short}, i.e. has length $\le d$, or {\it long} and contains a vertical edge at distance $\le d$ from $f$.
		The total number of possible short paths that contain $f$ is bounded.
		The number of long paths among $(p_i)$ that contain $f$ is bounded by the doubled number of vertical edges at distance $\le d$ from $f$.
	\end{proof}

	Take a subset $X$ in $\H^d$ and a point $x\in \H^d$.
	Let us call {\it a cone} the union of all geodesics between
	$x$ and $y$, where the union is taken over all $y\in X$. We denote this cone by   $\Cone (x, X)$.
	This definition can be also given for an arbitrary metric space $M$
	if we consider the union of all geodesics between $x$ and $y$, but we will use  this notion only for $M = \H^d$.
	
	For a set $X$ we denote by  $X_{\varepsilon}$ its $\varepsilon$-neighborhood.
	The standard volume in the space $\H^d$ is denoted by  $\Vol$,
	the volume of a ball of radius $\varepsilon$ is denoted by $\Vol(B(\cdot, \varepsilon))$.
	Clearly, this volume does not depend on the center of the ball, and in the notation above we write $\cdot$ without introducing the notation for the center. 
	
	\begin{lemma}[Linear bound for the  volume of neighborhoods of a cone]\label{le:cone}
		For all $\varepsilon > 0$ there exists  $C$, depending on $d$ and $\varepsilon$,
		such that if  $\gamma\in \H^d$  is a curve of length $l$ and $x$ is a point
		of $\gamma$, 
		then the volume of the $\varepsilon$-neighborhood of the cone satisfies  ${\rm Vol}(\Cone (x, \gamma)_{\varepsilon}) \leqslant C(l+1)$.
	\end{lemma}
	
	\begin{proof}
		Note that adding a geodesic interval to our curve, we obtain a closed curve of length at most $2l$. 
		Therefore
		we can assume that our curve in the assumption of the lemma is closed.
		
		If we move the point $x$ a small distance away from $\gamma$ and replace $\gamma$ by another curve that is close to $\gamma$ in the sense of the Hausdorff metric, then the original cone is contained in a small neighbourhood of the new one. 
		Therefore we can also assume that $x$ is at distance $\le \varepsilon$ from $\gamma$ but $x\not \in \gamma$. 
		We also assume that $\gamma$ is a union of geodesics. 
		Further we will use the estimation $d(x,\gamma) < l$.
		
		In the sequel we use the following notation. Given a continuous path $\kappa$, its length is denoted  by  $|\kappa|$.
		Given a metric space $M$ and $\delta > 0$, we 
		will use the notation
		$$
		N(\delta, M)= \min \{ n : \exists M'\subseteq M, \# M' = n, M \subseteq M'_{\delta}\},
		$$
		that is $N(\delta,M)$ is the minimal cardinality  of a $\delta$-net of $M$.
		
		Observe that for  $M \subset \mathbb{H}^d$ we have
		\begin{equation}\label{eq:coneball}
		\Vol(M_{\varepsilon}) \leq \Vol(B(\cdot, 2\varepsilon)) N(\varepsilon,M).
		\end{equation}
		The volume $\Vol(B(\cdot, 2\varepsilon))$ depends clearly only on  $\varepsilon$ and $d$.
		Therefore, it is sufficient to prove a linear upper bound for $N(\varepsilon, \Cone(x, \gamma))$.

		For a piecewise geodesic curve $\theta = y_1y_2\dots y_m$, (with $y_iy_{i+1}$ geodesic for all $i<m$) and $x \notin \theta$,  define the {\it angular length} of $\theta$  with respect to  $x$ as
		$$
		\Angle_x(\theta) = \angle y_1xy_2 +\dots + \angle y_{m-1}xy_m.
		$$
		
		Decompose the curve $\gamma$ into pieces $\gamma_0, \gamma_1, \gamma_2, \dots, \gamma_n$ in such a way that  $\Angle_x(\gamma_0) \leq \pi$ and
		$\Angle_x(\gamma_i) = \pi$ for  $ 1 \leq i \leq n$.
		
		It is clear that  $\Cone(x,\gamma)$ is the union  of cones $\cup_i \Cone(x, \gamma_i)$.
		Let us show that for some  $C$, 
		depending on $d$ and $\varepsilon$, we have
		\begin{equation}\label{eq:cone0}
		N(\varepsilon, \Cone(x, \gamma_0)) \leq C\cdot |\gamma| + 1,
		\end{equation}

		and
		for $1 \leq i \leq n, $
		\begin{equation}\label{eq:cone1}
		N(\varepsilon, \Cone(x, \gamma_i)) \leq C\cdot |\gamma_i| + 1.
		\end{equation}
		
		Moreover, we will show that the $\varepsilon$-nets that assure the inequalities 
		above can be chosen among $\varepsilon$-nets containing $x$. 
		This would imply the claim of the lemma, because the union of these  $\varepsilon$-nets is an $\varepsilon$-net for $\Cone(x,\gamma)$, and we have the estimation 
		\begin{equation}\label{eq:conesum}
		N(\varepsilon, \Cone(x,\gamma))\leq 1+ C\cdot|\gamma| + \sum_{i: 1\le i \le m}C\cdot |\gamma_i| \le 2Cl + 1.
		\end{equation}
		
		\noindent The $+1$ term corresponds to the point $x$ which belongs to each of our $\varepsilon$-nets.
		
		Now we will prove our statement for $\gamma_i$, we assume that it decomposes as a piecewise geodesic  $\gamma_i = y^i_0\dots y^i_{k^i}$.
		The cone  $\Cone(x, \gamma_i)$ is a union of  $k^i$ triangles.
		In the hyperbolic plane  $\mathbb{H}^2$  consider a polygon (possibly degenerate) consisting of the same triangles:
		$\Phi_i = XY^{i}_0Y^{i}_1\dots Y^{i}_{k^i}$ 
		is such that  the triangle $y^{i}_jxy^{i}_{j+1}$ 
		is isometric to the triangle $Y^{i}_jXY^{i}_{j+1}$ for all  $0 \leq j < k^i$.
		(Such $\Phi_i$ is what is called a {\it net}, not to be confused with $\varepsilon$-nets).
		The  perimeter of  $\Phi_0$ can be  estimated by 
		$d(x,y_0^0) + |\gamma_0| + d(x,y^0_{k^0})\leq 5|\gamma|$.
		For all  $i\ge 1$ observe that
		$\angle {Y}^i_0X{Y}^i_{k^i}=\pi$, 
		this implies that the union of  $Y^i_0X$ and $XY^i_{k^i}$ forms a geodesic,
		and hence  $|Y^i_0XY^i_{k^i}| \leq |\gamma_i|$, and therefore,  the perimeter $p_i$ of $\Phi_i$ is at most $2 |\gamma_i|$.
		
		By a linear isoperimetric inequality in $\H^2$ (see e. g. \cite{teufel} for a more general version), we know that a loop  $K$, of length $L$,  in a hyperbolic plane of constant curvature $-c$ is the boundary of  a  figure of volume (area in $\H^2$) at most  $L/\sqrt{c}$.
		
		Denote by  $\Phi'_{i}$ the set of points of  $\Phi_i$  that are at distance  greater than or equal to
		$\varepsilon/2$ from its boundary.
		Consider  an $(\varepsilon,\varepsilon)$-net $\mathcal{U}_i$ of $\Phi'_i$ (such a net exists by Remark \ref{re:separatednet}).
		Its cardinality $\#\mathcal{U}_i$ is at most  $\frac{p_i}{S\sqrt c}$, where  $S$ is the volume (the area in $\H^2$)  of a ball of radius $\varepsilon/2$.
		We can also find an $(\varepsilon/2)$-net  $\mathcal{W}_i$ of the boundary $\partial \Phi_i$ of  $\Phi_i$ such that $X\in \mathcal{W}_i$ and $\#\mathcal{W}_i \leq \frac{p_i}{\varepsilon/2} +1$.
		Observe that $\mathcal{W}_i$ is an $\varepsilon$-net of $\Phi_i \setminus \Phi'_i$.
		
		We get therefore an $\varepsilon$-net $\mathcal{U}_i \cup \mathcal{W}_i$ of $\Phi_i$ of cardinality at most  $2p_i/\varepsilon + \frac{p_i}{S\sqrt c} + 1$ points (including $X$).
		To show that there exists an  $\varepsilon$-net of $\Cone(x,\gamma_i)$ of the same cardinality,  observe the following.
		The polygon $\Phi_i$ is obtained by unfolding cone $\Cone(x,\gamma_i)$ of piecewise geodesic curve $\gamma_i$ of angular length at most $\pi$,
		then $\Phi_i$ maps to this cone by a map that does not increase distances.
		Assuming the estimations on $p_i$, we have proved inequalities \ref{eq:cone0} and \ref{eq:cone1} for
		
		\begin{equation}
		C = \frac{10}{\varepsilon} + \frac{5}{S\sqrt{c}}.
		\end{equation}

		Finally, from inequalities \ref{eq:coneball} and \ref{eq:conesum} we obtain
		
		\begin{equation}\label{eq:conefinal}
		{\rm Vol}(\Cone (x, \gamma)_{\varepsilon}) \leqslant \Vol(B(\cdot, 2\varepsilon)) \cdot (2Cl + 1).
		\end{equation} 
		
	\end{proof}

	In the claim of the Lemma below the word ``convex'' refers to adding paths between two points, not to Definition \ref{def:convex} of convexity
	with respect to an order. 
	Note that in contrast to the result of \cite{benjaminironen} the number of points (on our path) is not bounded.
	Given a set $X$ in a metric space $M$, define its {\it partial convex hull}  $\rm{PCH}(X)$ as the union of all geodesics between pairs of points in $X$.
	
	\begin{lemma}[Bound on the partial convex hull]\label{lem:PCH}
		\label{lm_union_geodes}
		There exists  $C$, depending on $d$ and $\varepsilon$,
		such that if  $\gamma \in \H^d$  is a curve of  length $l$, then
		the volume of the $\varepsilon$-neighborhood of its partial convex hull satisfies
		${\rm Vol} ({\rm PCH} ({\gamma})_{\varepsilon}) \leqslant C l+C$.	
	\end{lemma}

	\begin{proof}
		Fix some point $x\in \gamma$.
		Observe that by $\delta$-hyperbolicity of $\H^d$ we know that  ${\rm PCH}(\gamma)$ lies in the $\delta'$-neighborhood
		of the cone $\Cone (\gamma,x)$, for some $\delta'>0$. Therefore, its $\varepsilon$-neighborhood belongs to the $(\delta'+\varepsilon)$-neighborhood
		of the cone, and the claim follows therefore from the previous lemma.
	\end{proof}
	
	\begin{prop}\label{prop:hyperbolic}
		There exists a constant $K>0$ depending on $d$  such that
		for any
		order $T$ on $\Gamma_d$, with branches of all trees
		of $\overline{\Gamma_d}$
		convex with respect to $T$,
		we have  $\OR_{\Gamma_d, T}(k) \leqslant K$,  for any $k$.
	\end{prop}
	
	\begin{proof}
		For any finite subset $X$ of $\Gamma_{d}$ we can choose a hierarchical order of $\mathcal{F}_d$ and find a subset $X'\subset \mathcal{F}_d$ that is isometric to $X$ and that is ordered in the same way as $X$.
		So we assume that we have a finite subset $X \subset \mathcal{F}_{d} \subset \Gamma_d \subset \H^d$, of cardinality  $n+1$,  and assume that this subset admits a path of length $L$ in $\Gamma_d$
		visiting its points.
		Let us number the  points of the set $X$ with respect to  an order
		on $\mathcal{F}_d$, satisfying the assumption of the proposition: 
		$\alpha_0 >_T \alpha_1 >_T \dots >_T \alpha_n$. We want to show that
		$$
		\sum_{i=1}^{n} d_{\Gamma_d}(\alpha_{i-1}, \alpha_i) \le KL,
		$$
		for some constant $K$ depending only on $d$.
		
		Connect $\alpha_i$ with $\alpha_{i+1}$, $0 \le i \le n$ by a geodesic in
		$\Gamma_d$, and consider the union of these
		geodesics.
		We estimate the total number of vertices on 
		the obtained path
		using Lemma \ref{le:cone}
		about cones and its corollary (Lemma \ref{lem:PCH}) to partial convex hulls. 
		
		Now for each $i: 0 \le i \le n-1$ consider a standard up-and-down path $\gamma_i$ between $\alpha_{i}$ and $\alpha_{i+1}$:
		\[\alpha_i = u^i_0, u^i_1, u^i_2, \dots, u^i_{n_i} = \alpha_{i+1}.
		\]
		From Lemma \ref{le:up_and_down} we know that these paths are quasi-geodesics in $\Gamma_d$ with some constants $A,B$. 
		Hence they are also  quasi-geodesics in $\H^d$ with some constants $A'$ and $B'$ (depending only on  $d$).
		
		As we have already mentioned, the Morse Lemma for $\delta$-hyperbolic spaces implies that there exists $C_1$ depending on $A'$ and $B'$ (and the hyperbolicity constant of the space $\H^d$)
		such that any $(A',B')$-quasi-geodesic between two points in $\mathbb{H}^d$
		belongs to the $C_1$-neighborhood of any geodesic between these points, see e.g. Chapter 5 of \cite{ghysdelaharpe}. 
		As we have mentioned, $A',B'$, and therefore also $C_1$ depend only on $d$.
		
		Recall that we are discussing a finite set $X$, which admits a path of length $L$ in $\Gamma_d$ visiting all its points. 
		Since the vertices of $\Gamma_d$ form a uniformly discrete space,  quasi-isometrically embedded in $\H^d$, there is a piecewise geodesic path $\gamma_*$ of length at most $C_2 L$ in $\H^d$, visiting all points of $X$. Here the constant $C_2$ depend only on $d$. The constants $C_3,C_4 \dots$
		that we will choose later in the proof will also depend only on $d$.
		
		The geodesics joining  points $a_i$ and $a_{i+1}$ in $\mathbb{H}^d$ lie in ${\rm PCH}(\gamma_*)$. 
		All tiles corresponding to $u^i_j$ are in the $C_3$-neighborhood of ${\rm PCH}(\gamma_*)$, for some $C_3$. Constant  $C_3$ can be chosen as the sum of $C_1$ and the diameter of a tile.
		
		We can therefore apply Lemma \ref{lem:PCH} for $\gamma_*$ and $\varepsilon = C_3$ and find a constant $C_4$
		depending only on $d$ such that all tiles containing corresponding points $u^i_j$
		belong to a subset of volume at most $C_4(L+1)$ of $\H^d$. 
		
		Let $C_5$ be the volume of a tile in $\H^d$ (since  the tiles are isometric they have the same volume). 
		We can conclude that
		the number of distinct vertices in $\Gamma_d$ in the union of quasi-geodesics $\gamma_i$ is at most  $\frac{C_4}{C_5}(L+1)$.
		
		The observation above holds for any order, we did not use that branches are convex with respect to $T$. However, in the following observation this assumption  for $T$ is essential. 
		We recall that by Lemma \ref{lem:multiplicity} there exists a constant $C_6$ such that any point appears
		at most $C_6$ times among the vertices of the quasi-geodesics $\gamma_i$. This implies that  $l_{T}(X)$ 
		is at most
		$\frac{C_4C_6}{C_5}(L+1)$.

	\end{proof}

	As we have already mentioned, the result of Bonk and Schramm \cite{bonkschramm}
	implies  that any  $\delta$-hyperbolic metric space of bounded growth  embeds quasi-isometrically into $\H^d$ for some $d$.
	We do not explain the definition of bounded growth but recall that any graph of bounded degree is a particular case of this definition.
	
	Observe that if a space embeds  quasi-isometrically into $\H^d$, this space also embeds quasi-isometrically into $\Gamma_d$. Therefore, combining   Proposition \ref{prop:hyperbolic}  with
	Lemma \ref{le:pullback} we obtain
	
	\begin{thm} \label {thm:hyperbolicspace}
		Let $M$  be a $\delta$-hyperbolic graph of
		bounded degree. 
		Then there exists  an order $T$ and
		a constant $C$
		such that for all $k$
		$$
		\OR_{M,T}(k) \le C.
		$$
	\end{thm}
	
	In the formulation of the Theorem above we use the convention that edges of the graph
	are of length $1$, and we can assume that they are included in the metric space $M$, see Lemma \ref{lem:edgesnotedges}.
	
	\begin{rem}\label{rem:graphsqiHd}
		The order breakpoint of $\Gamma_{d}$ and the constant $C$ in Theorem \ref {thm:hyperbolicspace} of $\Gamma_d$ tend to infinity as $d$ tends to infinity. 
		To see this it is enough to consider uniform embeddings
		of a cube of dimension $d-1$ into $\Gamma_d$.
		For more on this see Section $6$
		in \cite{ErschlerMitrofanov2}.		
		
		From Lemma \ref{le:snakesqi} it follows that the same claim is clearly true for any uniformly discrete metric space, quasi-isometric to $\H^d$. 
		We remind the reader that
		for each $d\ge 1$ there exist finitely generated groups quasi-isometric to $\H^d$ \cites{Borel63, gromovpiatetski}, but it seems to be a challenging open problem  \cite{gromovMarkov} to understand 
		whether one can find more elementary constructions of such groups (for
		$d\ge 3$).
	\end{rem}
	
	\begin{rem}\label{rem:ahmedov}
		The properties of the $OR$ function we consider are not to be confused with the ``travelling salesman property''  of a very different nature,  introduced for finitely generated groups by Akhmedov 
		\cite{akhmedov1}, see also \cite{akhmedov2}, \cite{guba}. 
		While  hyperbolic groups (and spaces) are best possible cases for our question about $\OR(k)$, 
		hyperbolic groups are among worst possible case (``TS'' groups)
		for the question of Akhmedov, see \cite{akhmedov1}.
		We also  mention that all amenable groups are not so bad for that problem (they are not ``TS'' groups, see \cite{guba} ,
		where this observation is attributed to Thurston), while for the questions of our paper there is no visible  relation with amenability. Both amenable and non-amenable groups can  have finite $AN$-dimension (and thus satisfy logarithmic bound for $OR(k)$, as we will show in  \cite{ErschlerMitrofanov2}, see Theorem I)  and there are numerous classes of amenable groups satisfying  $OR(k)=k$ ( \cite{ErschlerMitrofanov2}, Corollary 6.8).	
	\end{rem}

	\subsection{Further examples with bounded order ratio functions} \label{Subsection:nothyperbolic}
	
	A metric space with bounded order ratio function does not need to be hyperbolic. 
	We have seen already in Section $3$ (see Lemma \ref{le:ORacyclicunion} and Remark \ref{rem:ORcyclicshift1}) that for cactus graphs the order ratio function is bounded. 
	Choosing circles of unbounded length we clearly can obtain not-hyperbolic examples (in particular among graphs of bounded degree). 
	
	A wider family of examples is provided by the following lemma.
	\begin{lemma} \label{lem:severalcopies}
		Let $M$ be a metric space that admits an order $T$ such that $\OR_{M,T}(k) \le C$ for all $k\ge 1$.  Let $\Omega$ be a subset of $M$.  Given $s\ge 1$, denote by $M_s(\Omega)$ a space which is obtained from $s$ copies of $M$ by gluing them over $\Omega$. 
		For any $s\ge 1$ the space $M_s(\Omega)$ admits an order $T_s$ such that $\OR_{M_s(\Omega),T_s}(k) \le sC + s-1$.
	\end{lemma}

	\begin{proof}
		For $1 \leq i \leq s$ we denote by $M_i$ the subset of $M_s(\Omega)$ that corresponds to the $i$-th copy of $M$.
		It is clear that $\Omega$ is the intersection of all $M_i$.
		We define the order $T_s$ on $M_s(\Omega)$ as follows: 
		\begin{enumerate}
			\item if $x\in M_1$, $y\not \in M_1$, then $x<_{T_s}y$;
			\item if $x, y \in M_1$, then we define $x<_{T_s}y$ if and only if $x<_T y$;
			\item if $x, y \not \in \Omega$, $x\in M_i$, $y\in M_j$ and $i < j$, then $x<_{T_s}y$;
			\item if $x,y\in M_i\setminus \Omega$ for some $i$ and $x<_T y$ , then $x<_{T_s}y$.
		\end{enumerate}
		
		Consider a finite set $X\subset M_s(\Omega)$ and denote by $L$ the length $l_{\opt}(X)$ of the shortest path  in $M_s(\Omega)$.
		We want to estimate $l_{T_s}(X)$ from above.  
		Let $X_1$ be the intersection $X\cap M_1$ and let $X_i$ for $i > 1$ be the intersection $X\cap (X_i \setminus \Omega)$.
		Without loss of generality we assume that all $X_i$ are non empty. 
		
		It is clear that 
		$$
		l_{T_s}(X) \leq (l_{T_s}(X_1) +\dots +l_{T_s}(X_s)) + (t_1 + \dots + t_{s-1}),
		$$
		where $t_i$ is the distance between the maximal point of $X_i$
		with respect to $T_s$ and the minimal  point of $X_{i+1}$
		with respect to $T_s$.
		
		For any $i$ we can estimate $l_{\opt}(X_i) \leq L$, $l_{T_s}(X_i) \leq C l_{\opt}(X_i) \leq CL$, $t_i \leq L$.
		Hence $l_{T_s}(X) \leq (sC + s-1)L$.
	\end{proof}
	
	We know that the assumption of Lemma \ref{lem:severalcopies} is in particular
	satisfied  when $M$ is a uniformly discrete hyperbolic space of uniformly bounded growth. 
	This assumption holds also  for any metric tree (not necessarily uniformly discrete).
	Figure  \ref{pic:nothyperbolic} shows an example of the construction above, where $M$ is a tree, $s=2$. Observe that the space in this example is non-hyperbolic. For all $r \in \mathbb{N}$ it contains an isometric copy of a cycle of length $2^{r+2}$.
	
	\begin{figure}[!htb] 
		\centering
		\includegraphics[scale=.55]{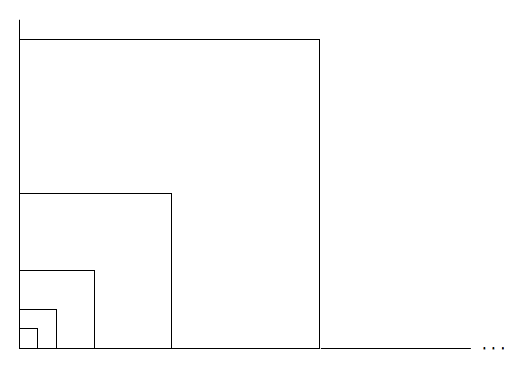}
		\caption{An example of a non-hyperbolic space with bounded order ratio function.
		}
		\label{pic:nothyperbolic}
	\end{figure}
	
	Other particular cases of the spaces from the claim of Lemma \ref{lem:severalcopies} are spaces obtained by gluing together several copies of graphs $\Gamma_d$ (which is quasi-isometric to $\H^{d}$), by a horizontal  subspace in
	the half-space model
	of dimension $d-1$. 
	We recall that in our terminology $\ORI_M$  
	is defined to be the equivalence class, with respect to equivalence up to a multiplicative constant, of functions  $\OR_{M'}$, where $M'$ is some $(\varepsilon, \delta)$-net of $M$.
	If we glue two copies of $\H^2$ by this horizontal subspace, we obtain a space $M$, satisfying $\ORI_M \le {\rm Const}$, but admitting contractible quasi-isometrically embedded circles of
	arbitrarily large length.
	
	We have discussed non-hyperbolic
	spaces with bounded function $\OR(k)$. It seems interesting to understand whether this can happen
	for finitely generated non-hyperbolic groups. In this context we ask the following questions: what is the asymptotics of $\OR(k)$
	for solvable Baumslag-Solitar groups? The same question can be asked for $\mathbb{Z}\wr \mathbb{Z}/2\mathbb{Z}$.
	
	\appendix
	
	\section{Proof of Lemma \ref{lem:edgesnotedges}}
	\label{app:edgesnotedges}
	
	Let $\Gamma$ be a graph (finite or infinite) with edges of length 1 and let $T$ be an order on the set $V$ of its vertices.
	We remind the reader that $\Star(T)$ is the order defined on $\Gamma$ (including its edges) described in Definition \ref{def:star}.
	
	We want  to prove that for  all
	$k$  it holds that
	$$
	\OR_{\Gamma, \Star(T)}(k) \leq 8  \OR_{V,T}(k) + 4.
	$$
	
	In the definition of the order   $\Star(T)$ we consider the partition of $\Gamma$ into star figures. 
	Consider a subset $X \subset \Gamma$.
	First observe that if  $X$ belongs  to one of the stars, then  $l_{\Star(T)}(X) \leq 2  l_{\opt}(X)$. This is because  the order on this star is  (a very particularly easy) example of a hierarchical
	order on a  tree.
	Also observe that if $X$ belongs to one of the edges of the graph, then  $l_{\Star(T)}(X) \leq 2 l_{\opt}(X)$. Indeed, the edge consists of two halves, in each of the halves the points
	are visited in a linear order, and the jump from one half to the other is not larger than the diameter of the set.
	In the sequel we can therefore assume that $X$ does not belong to one of the star figures and also is not inside one of the edges of the graph.
	
	Assume that $\#X \leq k+1$.
	The points of  $X$ belong to  $k_0 \leq k+1$ star figures.
	Denote by  $Y$ the set of centers of these stars, and put  $Z = X \bigcup Y$.
	It is clear that  $l_{\Star(T)}(X) \leq l_{\Star(T)}(Z)$.
	We recall that by definition of our order the path first visits all points of one of the star figures, then all points from another one and so on.
	We can therefore write 
	$$
	l_{\Star(T)}(Z) = l_1 + l_2,
	$$
	where  $l_1$ is the sum of lengths of jumps between distinct star figures (the number of such jumps  is $k_0-1$),  and $l_2$  is the total length of  jumps inside star figures.
	
	First we estimate  $l_1$. The length of the jump from the star figure  $S_1$ centered at  $y_1$ to the star figure  $S_2$ centered at  $y_2$ is at most 
	$d(y_1, y_2) + 1 \leq 2 d(y_1,y_2)$. 
	This implies that the total sum of such jumps is not  greater than the length of the associated path in $Y$, that is 
	$$
	l_1 \leq 2 l_T(Y) \leq 2\OR_{V,T}(k_0-1)\cdot l_{\opt}(Y).
	$$
	Let us understand the relation between  $l_{\opt}(Y)$ and $l_{\opt}(X)$.
	If we have some  path through the points of  $X$, we can modify it to visit also all points of  $Y$: after the first visit of each star figure, except of the first and last visited star figures, we jump to its center and back after this first visit.  In the first visited star figure, we start with its center and then continue
	as in the original path.
	In the last visited star figure, after the last visit we jump to its center (and do not jump back).
	
	The length is increased by at most
	$1$ for each star figure, and not by more than $1/2$ in case of the first visited star figure and the last visited star figure. Therefore,
	$$
	l_{\opt}(Y) \leq l_{\opt}(X) + (k_0 - 1).
	$$
	For the optimal path of $X$ consider for each visited star figure (except the first one) the first visited point. 
	These $k_0-1$ points are  distinct middle points of some edges. 
	Since the distance between two distinct middle points is at least $1$, we observe that
	$$
	l_{\opt}(X) \geq k_0-2.
	$$
	Summing this with the inequality  $2 l_{\opt}(X) \geq 1$ (which holds because of our assumption that $X$ is not inside one star figure and not inside one edge), we get that 
	$k_0 - 1 \leq 3 l_{\opt}(X)$.
	And, therefore,  $l_1$ satisfies
	$l_1 \leq  8 \OR_{V,T}(k)l_{\opt}(X)$.
	
	Now we estimate  $l_2$.  Let us look at the position of the  points of $X$  on the corresponding  edges  of the graph.
	Suppose that  $x \in X$ belongs to the edge $e$. We denote by $m_e$ the middle point of this edge, and suppose that $x$ belongs to 
	the  half-edge $m_e v$,
	$v \in V$. We suppose also that   the interval $m_e x$ does not
	have other points of  $X$. In this case  we associate to  $x$ the number  $d(x,v)$. In other words, we associate the distance from $x$ to the vertex $v$ if $x$ is the most distant point from the vertex  $v$ on the half-edge $m_ev$.
	Otherwise, we do not associate anything to $x$.
	Observe that if  $S$ is the total sum of associated numbers, then 
	$l_2 \leq 2S$ because of the hierarchical order on each star figure.
	Note that for each edge there are at most two associated numbers. Let us call a point  of $X$ {\it leading} if it has an associated number and it is the largest of the two numbers associated
	to the points of  its edge. The set of leading points is denoted by  $X_{r}$.
	The sum of the numbers associated to leading points clearly satisfies
	$S' \geq S/2$.
	If we have a leading point with an associated number  $s_1$, and another leading point with an associated number  $s_2$,  then the distance  between these two points is at least 
	$s_1+s_2$.
	Taking in account that $X$ does not belong to one edge and hence  $\# X_r>1$, we get 
	$l_{\opt}(X) \geq l_{\opt}(X_r) \geq S'$,
	and consequently, 
	$
	l_2 \leq 4l_{\opt}(X)$.
	Therefore,
	$$
	l_{\Star(T)}(X) \leq l_{\Star(T)}(Z) = l_1+l_2 \leq (8 \OR_{V,T}(k) + 4)l_{\opt}(X),
	$$
	
	this concludes the proof of Lemma \ref{lem:edgesnotedges}.
	\qed
	
	\section{A theorem of Kapovich about tripods in geodesic
		metric spaces and related questions}
	\label{sec:tripods}
	
	\begin{definition}
		Let $M$ be a metric space.
		An $(R,\varepsilon)$-tripod is the union of an $\varepsilon$-geodesic $\mu=YO \cup OZ$ (a curve of length at most $d_M(Y,Z)+\varepsilon$) and a
		curve $\gamma = XO$ of length at most $d+\varepsilon$, where $d$ is the distance from $X$ to $\mu$, 
		where $d_M(O,X), d_M(O,Y), d_M(O,Z) \ge R$. 
	\end{definition}

	We do not recall the definition of \emph{path metric spaces} (a generalisation of a notion of geodesic metric spaces, which makes sense for spaces which are not necessarily complete), but mention that any geodesic metric space is an example  of  a path metric space.
	
	We use the following result of M.~Kapovich.
	
	\begin{claim}\label{kapovich}
		[\cite{kapovich}, Theorem 1.4, see also Corollary 7.3]
		Let $M$ be a path metric space not quasi-isometric to a point, a ray or a line. Then for any $R>0$ and any $\varepsilon>0$, $M$ admits an $(R,\varepsilon)$-tripod.
	\end{claim}
	
	\noindent Moreover, Kapovich showed that for any $R>0$ one can find a $(R,0)$-tripod inside any ultralimit $M_{\omega}$ of the constant sequence of (pointed) metric spaces $M$ from the theorem.
	Since $M_{\omega} = M$ for proper geodesic metric spaces, such metric spaces (if not quasi-isometric to a point a ray or a line) admit tripods with $\mu$ being a geodesic segment and $\gamma$ being a geodesic segment of shortest length to $\mu$.

	Now we recall our notation for tripods:  a tripod $T_R$ consists of $3$ segments
	of length $R$ glued together by an endpoint.
	\begin{claim}\label{le:Bclaim1}
		Let $M$ be a geodesic metric space and let $T$ be a tripod with center $y'$ and legs $y'A'$, $y'B'$ and $y'x'$.
		Let $\mu: T \to M$ be a mapping and denote $\mu(A') = A$, $\mu(B') = B$, $\mu(y') = y$, $\mu(x') = x$.		 		
		Assume that $\mu:[A',y'] \cup [y',B'] \to M$ is  a length parametrized curve such that 
		$$
		d_M(\mu(A'), \mu(B')) \leq d_T(A',y') + d_T(B',y') -\varepsilon
		$$
		for  $\varepsilon > 0$.
		Assume that $\mu([y'x']) \to M$ is a length parametrized geodesic curve, and $d_M(x',y') \leq d + \varepsilon$,  
		where $d$ is the distance
		from $x$ to $\mu([A',B'])$.
		
		Then the map 
		$\mu:T\to M$ satisfies 
		$$
		\frac{1}{3} d_{T}(z,t) - \varepsilon \le d_M(\mu(z), \mu(t)) \le 
		d_{T}(z,t),
		$$
		
		for any $z,t \in T$.
	\end{claim}
	
	\begin{tikzpicture}[scale = 0.8]
	\begin{scope}[shift={(8,0)}]
	\draw [ultra thick] plot [smooth] coordinates {(-2,0) (-1,-0.25) (0.15,0) (1,0.3) (2,0.15)};
	\draw [thick] (-0.5,1.5) to[out=-60,in=100] node[midway,above] {} (0.15,0);
	\node[above] at (-2, 0)[above] {$A$};
	\node[above] at (2, 0.15)[above] {$B$};
	\node[below] at (0.15, 0)[below] {$y$};
	\node[above] at (-0.5, 1.5)[above] {$x$};
	\node[above] at (0.5, 1)[above] {$M$};
	\end{scope}
	
	\draw[->][thick] (3,0.5) -- node[midway,above] {$\mu$} (4,0.5);
	\draw[thick] (-2,-1) -- (0,0) -- (2,-1);
	\draw[thick] (0,0) -- (0,1.5);
	\node[above] at (-2, -1)[above] {$A'$};
	\node[above] at (2, -1)[above] {$B'$};
	\node[above] at (0, 0)[below] {$y'$};
	\node[above] at (0, 1.5)[above] {$x'$};
	\node[above] at (-1, 1)[above] {$T$};
	\end{tikzpicture}
	
	\begin{proof}
		It is clear that for any $z,t \in T$ that
		correspond to the leg $[x',y']$ we have $d_M(\mu(z), \mu(t))= d_T(z,t)$. 
		We also know that for any $z,t \in T$ such that $z, t \in A'B'$
		(that is, in the union of two legs)
		it holds that
		$$
		d_T(z,t) -\varepsilon \le d_M(\mu(z), \mu(t)) \le d_T(z,t).
		$$
		It is sufficient therefore to consider the case when the point $z \in [y',x']$ and the other point $t \in [A',B']$. 
		
		Observe that in this case
		$$
		d_T(z,t) = d_T(z,y')+ d_T(y',t)
		$$
		and hence
		$$
		d_M(\mu(z), \mu(t)) \le d_M(\mu(z), \mu(y'))+ d_M(\mu(t), \mu(y')) \le
		d_T(z,y')+ d_T(y',t) = d_T(z,t).
		$$
		
		Denote $d_T(t,y')$ by $a$ and denote $d_T(z,y')$ by $b$.
		It is clear that $d_T(z,t) = a+b$.
		
		If $a \geqslant 2b$, then $a-b \geq \frac{1}{3}(a+b)$ and
		
		$$
		d_M(\mu(z),\mu(t)) \ge
		|d_M(\mu(t),y)
		-d_M(\mu(z),y)| \ge a -\varepsilon - b \geqslant d_T(z,t)/3 - \varepsilon. 
		$$
		
		If $a\leqslant 2b$, then $b \geqslant \frac{1}{3}(a+b)$ and
		
		$$
		d_M(\mu(z),\mu(t)) \ge
		d_M(\mu(z),\mu([A',B']) \ge b-\varepsilon \ge d_T(z,t)/3 - \varepsilon.
		$$
	\end{proof}
	
	Using Claim \ref{kapovich} (taking in account a remark   below this claim) and Claim \ref{le:Bclaim1} we conclude that if $M$ is a geodesic metric space, not quasi-isometric to a point, a ray or a line, then 
	for any $\varepsilon>0$ and each  $n\ge 1$ there exists $\rho_n: T_n \to M$ such that for any $x,y \in T_n$
	$$
	\frac{1}{3} d_{T_n}(x,y) - \varepsilon \le d_M(\rho_n(x), \rho_n(y)) \le 
	d_{T_n}(x,y) +\varepsilon.
	$$
	Thus we have proved Lemma \ref{lem:containstripod}.
	
	\begin{rem} It is clear that the constant $1/3$ in Claim \ref{le:Bclaim1} can not be improved, as an example of a circle shows (or, if one wants to avoid degenerate tripods, an example of an interval attached by its endpoint to a circle)
	\end{rem}
	
	\begin{tikzpicture}[scale = 1.5]
	\draw (0,1) -- (1,1);
	\draw[ultra thick] (0,1.1) -- (0,0) -- (1,0) -- (1,2);
	\draw[ultra thick] (1,0) -- (2,0);
	\node[above] at (0, 1)[left] {$\mu(t)$};
	\node[above] at (1, 0)[below] {$y$};
	\node[above] at (1, 1)[right] {$\mu(z)$};
	\node[above] at (1, 1.5)[right] {};
	\node[above] at (1, 0.5)[right] {};
	\node[above] at (0.5, 0)[below] {};
	\node[above] at (1.5, 0)[below] {};
	\node[above] at (0, 0.5)[left] {};
	\end{tikzpicture}


\begin{thebibliography}{}
		
		\bib{akhmedov1}{article}{
			author={Akhmedov, Azer},
			title={Travelling salesman problem in groups},
			conference={
				title={Geometric methods in group theory},
			},
			book={
				series={Contemp. Math.},
				volume={372},
				publisher={Amer. Math. Soc., Providence, RI},
			},
			date={2005},
			pages={225--230},
			review={\MR{2140092}},
			doi={10.1090/conm/372/06888},
		}
		
		\bib{akhmedov2}{article}{
			author={Akhmedov, Azer},
			title={Perturbation of wreath products and quasi-isometric rigidity. I},
			journal={Int. Math. Res. Not. IMRN},
			date={2008},
			number={4},
			pages={Art. ID rnn004, 18},
			issn={1073-7928},
			review={\MR{2424177}},
			doi={10.1093/imrn/rnn004},
		}
		
		
		
		\bib{benjaminironen}{article}{
			author={Benjamini, Itai},
			author={Eldan, Ronen},
			title={Convex hulls in the hyperbolic space},
			journal={Geom. Dedicata},
			volume={160},
			date={2012},
			pages={365--371},
			issn={0046-5755},
			review={\MR{2970060}},
			doi={10.1007/s10711-011-9687-8},
		}
		
		
		
		\bib{bartholdiplatzman82}{article}{
			author={Bartholdi, J. J., III},
			author={Platzman, L. K.},
			title={An $O(N\,{\rm log}\,N)$ planar travelling salesman heuristic based
				on spacefilling curves},
			journal={Oper. Res. Lett.},
			volume={1},
			date={1982},
			number={4},
			pages={121--125},
			issn={0167-6377},
			review={\MR{687352}},
			doi={10.1016/0167-6377(82)90012-8},
		}
		
		\bib{bartholdiplatzman89}{article}{
			author={Platzman, Loren K.},
			author={Bartholdi, John J., III},
			title={Spacefilling curves and the planar travelling salesman problem},
			journal={J. Assoc. Comput. Mach.},
			volume={36},
			date={1989},
			number={4},
			pages={719--737},
			issn={0004-5411},
			review={\MR{1072243}},
			doi={10.1145/76359.76361},
		}
		
		\bib{bertsimasgrigni}{article}{
			author={Bertsimas, Dimitris},
			author={Grigni, Michelangelo},
			title={Worst-case examples for the spacefilling curve heuristic for the
				Euclidean traveling salesman problem},
			journal={Oper. Res. Lett.},
			volume={8},
			date={1989},
			number={5},
			pages={241--244},
			issn={0167-6377},
			review={\MR{1022839}},
			doi={10.1016/0167-6377(89)90047-3},
		}
		
		\bib{bhalgatetal}{article}{
			author={Bhalgat, Anand},
			author={Chakrabarty, Deeparnab},
			author={Khanna, Sanjeev},
			title={Optimal lower bounds for universal and differentially private
				Steiner trees and TSPs},
			conference={
				title={Approximation, randomization, and combinatorial optimization},
			},
			book={
				series={Lecture Notes in Comput. Sci.},
				volume={6845},
				publisher={Springer, Heidelberg},
			},
			date={2011},
			pages={75--86},
			review={\MR{2863250}},
		}
		
		
		\bib{bonkschramm}{article}{
			author={Bonk, M.},
			author={Schramm, O.},
			title={Embeddings of Gromov hyperbolic spaces},
			journal={Geom. Funct. Anal.},
			volume={10},
			date={2000},
			number={2},
			pages={266--306},
			issn={1016-443X},
			review={\MR{1771428}},
			doi={10.1007/s000390050009},
		}
		
		\bib{Borel63}{article}{
			author={Borel, Armand},
			title={Compact Clifford-Klein forms of symmetric spaces},
			journal={Topology},
			volume={2},
			date={1963},
			pages={111--122},
			issn={0040-9383},
			review={\MR{0146301}},
			doi={10.1016/0040-9383(63)90026-0},
		}
		
		
		\bib{BridsonHaefliger}{book}{
			author={Bridson, Martin R.},
			author={Haefliger, Andr\'{e}},
			title={Metric spaces of non-positive curvature},
			series={Grundlehren der Mathematischen Wissenschaften [Fundamental
				Principles of Mathematical Sciences]},
			volume={319},
			publisher={Springer-Verlag, Berlin},
			date={1999},
			pages={xxii+643},
			isbn={3-540-64324-9},
			review={\MR{1744486}},
			doi={10.1007/978-3-662-12494-9},
		}
		
		\bib{Christodoulou}{article}{
			author={Christodoulou, George},
			author={Sgouritsa, Alkmini},
			title={An improved upper bound for the universal TSP on the grid},
			conference={
				title={Proceedings of the Twenty-Eighth Annual ACM-SIAM Symposium on
					Discrete Algorithms},
			},
			book={
				publisher={SIAM, Philadelphia, PA},
			},
			date={2017},
			pages={1006--1021},
			review={\MR{3627793}},
			doi={10.1137/1.9781611974782.64},
		}
		
		
		\bib{drutukapovichbook}{book}{
			author={Dru\c{t}u, Cornelia},
			author={Kapovich, Michael},
			title={Geometric group theory},
			series={American Mathematical Society Colloquium Publications},
			volume={63},
			note={With an appendix by Bogdan Nica},
			publisher={American Mathematical Society, Providence, RI},
			date={2018},
			pages={xx+819},
			isbn={978-1-4704-1104-6},
			review={\MR{3753580}},
		}
		
		
		\bib{dunwoody}{article}{
			author={Dunwoody, M. J.},
			title={The accessibility of finitely presented groups},
			journal={Invent. Math.},
			volume={81},
			date={1985},
			number={3},
			pages={449--457},
			issn={0020-9910},
			review={\MR{807066}},
			doi={10.1007/BF01388581},
		}
		
		
		
		\bib{eades2}{article}{
			author = {Eades, Patrick},
			author= {Mestre, Juli\'{a}n},
			title = {An optimal lower bound for hierarchical universal solutions for TSP on the plane},
			journal = {Computing and combinatorics. 26th international conference, COCOON 2020, Atlanta, GA, USA, August 29--31, 2020. Proceedings},
			pages = {222--233},
			date = {2020},
			publisher = {Cham: Springer},
		}
		
		\bib{ErschlerMitrofanov2}{article}{
			author={Erschler, Anna},
			author={Mitrofanov, Ivan},
			title={Assouad-Nagata dimension and gap for ordered metric spaces},
			status={to appear in Commenarii Mathematici Helvetici, https://arxiv.org/abs/2109.12181}
		}
		
		
		\bib{ghysdelaharpe}{collection}{
			title={Sur les groupes hyperboliques d'apr\`es Mikhael Gromov},
			language={French},
			series={Progress in Mathematics},
			volume={83},
			editor={Ghys, {\'E}.},
			editor={de la Harpe, P.},
			note={Papers from the Swiss Seminar on Hyperbolic Groups held in Bern,
				1988},
			publisher={Birkh\"auser Boston Inc.},
			place={Boston, MA},
			date={1990},
			pages={xii+285},
			isbn={0-8176-3508-4},
			review={\MR{1086648 (92f:53050)}},
		}
		
		\bib{gorodezkyetal}{article}{
			author={Gorodezky, Igor},
			author={Kleinberg, Robert D.},
			author={Shmoys, David B.},
			author={Spencer, Gwen},
			title={Improved lower bounds for the universal and {\it a prior}i TSP},
			conference={
				title={Approximation, randomization, and combinatorial optimization},
			},
			book={
				series={Lecture Notes in Comput. Sci.},
				volume={6302},
				publisher={Springer, Berlin},
			},
			date={2010},
			pages={178--191},
			review={\MR{2755835}},
		}
		
		\bib{gouezelschur}{article}{
			author={Gou\"ezel, S\'ebastien},
			author={Shchur, Vladimir},
			title={A corrected quantitative version of the Morse lemma},
			journal={J. Funct. Anal.},
			volume={277},
			date={2019},
			number={4},
			pages={1258--1268},
			issn={0022-1236},
			doi={https://doi.org/10.1016/j.jfa.2019.02.021},
		}
		
		
		\bib{gromovpiatetski}{article}{
			author={Gromov, M.},
			author={Piatetski-Shapiro, I.},
			title={Nonarithmetic groups in Lobachevsky spaces},
			journal={Inst. Hautes \'{E}tudes Sci. Publ. Math.},
			number={66},
			date={1988},
			pages={93--103},
			issn={0073-8301},
			review={\MR{932135}},
		}
		
		
		\bib{gromovMarkov}{article}{
			author={Gromov, M.},
			title={Hyperbolic dynamics, Markov partitions and Symbolic Categories, Chapters $1$ and $2$},
			journal={preprint 2016,  https://www.ihes.fr/~gromov/wp-content/uploads/2018/08/SymbolicDynamicalCategories.pdf}
			
		}
		
		
		\bib{manning}{article}{
			author={Groves, D.},
			author = {Manning, J.F.},
			title={Dehn filling in relatively hyperbolic groups.},
			journal={ Isr. J. Math.},
			number={168},
			pages={317--429},
			date={2008}
		}
		
		
		\bib{guba}{article}{
			author={Guba, V. S.},
			title={Traveller salesman property and Richard Thompson's group $F$},
			conference={
				title={Topological and asymptotic aspects of group theory},
			},
			book={
				series={Contemp. Math.},
				volume={394},
				publisher={Amer. Math. Soc., Providence, RI},
			},
			date={2006},
			pages={137--142},
			review={\MR{2216711}},
			doi={10.1090/conm/394/07439},
		}
		
		
		\bib{HKL}{article}{
			author={Hajiaghayi, Mohammad T.},
			author={Kleinberg, Robert},
			author={Leighton, Tom},
			title={Improved lower and upper bounds for universal TSP in planar metrics},
			conference={
				title={Proceedings of the Seventeenth Annual ACM-SIAM Symposium on
					Discrete Algorithms},
			},
			book={
				publisher={ACM, New York},
			},
			date={2006},
			pages={649--658},
			review={\MR{2368861}},
			doi={10.1145/1109557.1109628},
		}
		
		
		\bib{kapovich}{article}{
			author={Kapovich, Michael},
			title={Triangle inequalities in path metric spaces},
			journal={Geom. Topol.},
			volume={11},
			date={2007},
			pages={1653--1680},
			issn={1465-3060},
			review={\MR{2350463}},
			doi={10.2140/gt.2007.11.1653},
		}
		
		
		
		\bib{Nima}{article}{
			author={Hoda, Nima},
			title={Shortcut graphs and groups},
			journal={Trans. Amer. Math. Soc.},
			volume={375},
			date={2022},
			number={4},
			pages={2417--2458},
			issn={0002-9947},
			review={\MR{4391723}},
			doi={10.1090/tran/8555},
		}
		
		
		
		
		
		
		\bib{jiadoubling}{article}{
			author={Jia, Lujun},
			author={Lin, Guolong},
			author={Noubir, Guevara},
			author={Rajaraman, Rajmohan},
			author={Sundaram, Ravi},
			title={Universal approximations for TSP, Steiner tree, and set cover},
			conference={
				title={STOC'05: Proceedings of the 37th Annual ACM Symposium on Theory
					of Computing},
			},
			book={
				publisher={ACM, New York},
			},
			date={2005},
			pages={386--395},
			review={\MR{2181640}},
			doi={10.1145/1060590.1060649},
		}
		
		
		\bib{justin71}{article}{
			author={Justin, Jacques},
			title={Groupes et semi-groupes \`a croissance lin\'{e}aire},
			language={French},
			journal={C.~R.~Acad. Sci. Paris S\'{e}r. A-B},
			volume={273},
			date={1971},
			pages={A212--A214},
			issn={0151-0509},
			review={\MR{289689}},
		}
		
		
		
		\bib{KrauthgamerLee}{article}{
			author={Robert Krauthgamer},
			author={James Lee},
			title={Algorithms on negatively curved spaces},
			journal={47th Annual IEEE Symposium on Foundation of Computer Science (FOCS'06),
				date={2006},
				pages={119-132}
			},
		}
		
		
		
		\bib{MannHowgroupsgrow}{book}{
			author={Mann, Avinoam},
			title={How groups grow},
			series={London Mathematical Society Lecture Note Series},
			volume={395},
			publisher={Cambridge University Press, Cambridge},
			date={2012},
			pages={x+199},
			isbn={978-1-107-65750-2},
			review={\MR{2894945}},
		}
		
		\bib{Schalekamp}{article}{
			author={Schalekamp, Frans},
			author={Shmoys, David B.},
			title={Algorithms for the universal and a priori TSP},
			journal={Oper. Res. Lett.},
			volume={36},
			date={2008},
			number={1},
			pages={1--3},
			issn={0167-6377},
			review={\MR{2377182}},
			doi={10.1016/j.orl.2007.04.009},
		}
		
		
		\bib{schur}{article}{
			author={Shchur, Vladimir},
			title={A quantitative version of the Morse lemma and quasi-isometries
				fixing the ideal boundary},
			journal={J. Funct. Anal.},
			volume={264},
			date={2013},
			number={3},
			pages={815--836},
			issn={0022-1236},
			review={\MR{3003738}},
			doi={10.1016/j.jfa.2012.11.014},
		}
		
		\bib{serreTrees}{book}{
			author={Serre, Jean-Pierre},
			title={Trees},
			series={Springer Monographs in Mathematics},
			note={Translated from the French original by John Stillwell;
				Corrected 2nd printing of the 1980 English translation},
			publisher={Springer-Verlag, Berlin},
			date={2003},
			pages={x+142},
			isbn={3-540-44237-5},
			review={\MR{1954121}},
		}
		
		
		\bib{teufel}{article}{
			author={Teufel, Eberhard},
			title={A generalization of the isoperimetric inequality in the hyperbolic
				plane},
			journal={Arch. Math. (Basel)},
			volume={57},
			date={1991},
			number={5},
			pages={508--513},
			issn={0003-889X},
			review={\MR{1129528}},
			doi={10.1007/BF01246751},
		}
		
		
	\end{thebibliography}
\end{document}